\documentclass[reqno,a4paper,11pt]{article}
\usepackage{amssymb,amsmath}
\usepackage{amsthm,eucal}
\usepackage{amsfonts}
\usepackage{mathrsfs,enumerate}
\usepackage[usenames]{color}
\usepackage[
color,
final] 
{showkeys}
\usepackage{graphicx}

\definecolor{refkeybis}{gray}{.65}
\definecolor{labelkeybis}{gray}{.65}
{\makeatletter
\def\SK@refcolor{\color{refkeybis}}%
\def\SK@labelcolor{\color{labelkeybis}}}

\numberwithin{equation}{section} 
\newtheorem{theorem}{Theorem}[section] 
\newtheorem{lemma}[theorem]{Lemma}
\newtheorem{corollary}[theorem]{Corollary}
\newtheorem{proposition}[theorem]{Proposition}

\theoremstyle{definition}
\newtheorem{remark}[theorem]{Remark}

\newtheorem{definition}[theorem]{Definition}

\newcommand{\R}{\mathbb{R}}

\newcommand{\N}{\mathbb{N}}

\newcommand{\DD}{\mathscr{D}}

\newcommand{\MM}{\mathcal{M}}
\newcommand{\HH}{\mathscr{H}}

\newcommand{\VV}{{\mathscr{V}}}
\newcommand{\UU}{{\mathscr{U}}}
\newcommand{\WW}{\mathscr{W}}

\newcommand{\ii}{{\mbox{\boldmath$i$}}}

\newcommand{\nn}{{\mbox{\boldmath$n$}}}
\newcommand{\snn}{{\mbox{\scriptsize\boldmath$n$}}}
\newcommand{\rr}{{\mbox{\boldmath$r$}}}

\renewcommand{\ss}{{\mbox{\boldmath$s$}}}

\newcommand{\vv}{{\mbox{\boldmath$v$}}}
\newcommand{\svv}{{\mbox{\scriptsize\boldmath$v$}}}
\newcommand{\ww}{{\mbox{\boldmath$w$}}}

\newcommand{\tauV}{{\kern-3pt\tau}}

\newcommand{\nnu}{{\mbox{\boldmath$\nu$}}}

\renewcommand{\restriction}[1]{\lower3pt\hbox{$|_{#1}$}}

\newcommand{\de}{\partial} 
\newcommand{\supp}{\mathop{\rm supp}\nolimits} 

\newcommand{\Hess}{\mathop{\rm Hess}\nolimits}

\newcommand{\Leb}[1]{{\mathscr L}^{#1}} 

\newcommand{\weakto}{\rightharpoonup}
\newcommand{\eps}{{\varepsilon}}
\renewcommand{\to}{\rightarrow}

\newcommand{\Prob}[2][]{{\mathscr P}_{#1}(#2)} 
\newcommand{\Probac}[2][]{{\mathscr P}_{\!\!ac}(#2)}

\newcommand{\admissibleplan}[2]{\Gamma (#1,#2)}
\newcommand{\te}{\textrm}
\newcommand{\veps}{\varepsilon}
\definecolor{lblue}{rgb}{0.5,0.5,1}

\newcommand{\topref}[2]{\stackrel{\eqref{#1}}#2}
\newcommand{\Wmod}{\mathcal W}
\newcommand{\W}{\mathcal W}
\newcommand{\Mob}{{\mathsf m}}
\newcommand{\med}{{M_{\uparrow}}}
\renewcommand{\d}{{\mathrm d}}
\newcommand{\dx}{\d x}
\newcommand{\dt }{\d t}
\newcommand{\Rd}{{\R^d}}
\newcommand{\loc}{{\mathrm{loc}}}
\newcommand{\comp}{{\mathrm c}}
\newcommand{\CE}{{\mathcal{CE}}}
\newcommand{\tif}{\text{if }}
\newcommand{\sfS}{\mathsf{S}}
\newcommand{\DenseDom}{D}
\newcommand{\GMC}[2]{GMC(#1,#2)}
\newcommand{\restr}[1]{\lower3pt\hbox{$|_{#1}$}}
\newcommand{\mass}{\mathfrak m}
\newcommand{\nchi}{{\raise.3ex\hbox{$\chi$}}}
\newcommand{\GF}{\mathcal S}
\newcommand{\Semi}{S}

\makeindex
\begin{document}

\title{Nonlinear mobility continuity equations and generalized displacement convexity
}

\author{J. A. Carrillo\thanks{ICREA (Instituci\'o Catalana de
Recerca i Estudis Avan\c cats) and Departament de Matem\`atiques,
Universitat Aut\`onoma de Barcelona, E-08193 Bellaterra, Spain.
E-mail: \textsf{carrillo@mat.uab.es}.}, S.
Lisini\thanks{Dipartimento di Matematica ``F. Casorati'', Universit\`a
degli Studi di Pavia, Italy. E-mail:
\textsf{stefano.lisini@unipv.it}.}, G. Savar\'e
\thanks{Dipartimento di Matematica ``F. Casorati'', Universit\`a
degli Studi di Pavia, Italy. E-mail:
\textsf{giuseppe.savare@unipv.it}.}, D.
Slep\v{c}ev\thanks{Department of Mathematical Sciences, Carnegie
Mellon University, Pittsburgh, USA. E-mail:
\textsf{slepcev@math.cmu.edu}.}}

\date{\today}

\maketitle

\begin{abstract} We consider the geometry of the space of Borel
measures endowed with a distance that is defined by generalizing
the dynamical formulation of the Wasserstein distance to concave,
nonlinear mobilities. We investigate the energy landscape of
internal, potential, and interaction energies. For the internal
energy, we give an explicit sufficient condition for geodesic
convexity which generalizes the condition of McCann. We take an
eulerian approach that does not require global information on the
geodesics. As by-product, we obtain existence, stability, and
contraction results for the semigroup obtained by solving the
homogeneous Neumann boundary value problem for a nonlinear
diffusion equation in a convex bounded domain. For the potential
energy and the interaction energy, we present a non-rigorous
argument indicating that they are not displacement semiconvex.
\end{abstract}


\section{Introduction}

\paragraph{Displacement convexity and Wasserstein distance.}
In \cite{McCann97}, McCann introduced the notion of
\emph{displacement convexity} for integral functionals of the form
\begin{equation*}\label{i:func}
\UU(\mu):= \int_\Omega U(\rho(x))\,\dx \quad \text{if } \mu=\rho
\Leb{d}, \qquad \text{$U:[0,+\infty)\to \R$ is a convex function,}
\end{equation*}
defined on the set $\Probac{\Omega}$ of the Borel probability
measures in a convex open domain $\Omega\subset \R^d$, which are
absolutely continuous with respect to the Lebesgue measure $\Leb
d$. Displacement convexity of $\UU$ means convexity along a
particular class of curves, given by \emph{displacement
interpolation} between two given measures. These curves turned out
to be the geodesics of the space $\Probac{\Omega}$ endowed with
the euclidean Wasserstein distance.

We recall that the Wasserstein distance $W$ between two Borel
probability measures $\mu_0$ and $\mu_1$ on $\Omega$ is defined by
the following optimal transportation problem (Kantorovitch relaxed version)
\begin{equation*}\label{defwi}
W^2(\mu_0,\mu_1):=\min \left\{\int _{\Omega \times \Omega}
|x-y|^2\,\d\gamma (x,y): \gamma \in
\admissibleplan{\mu_0}{\mu_1}\right \},
\end{equation*}
where $\admissibleplan{\mu_0}{\mu_1}$ is the set of admissible
plans/couplings between $\mu_0$ and $\mu_1$, that is the set of
all Borel probability measures on $\Omega \times \Omega$ with
first marginal $\mu_0$ and second marginal $\mu_1$.

We introduce the ``pressure'' function $P$, defined by
\begin{equation}
\label{eq:10} P(r):=r U'(r)-\big(U(r)-U(0)\big)=
\int_0^r sU''(s)\,\d s\quad \text{so that}\quad P'(r)=rU''(r),\
P(0)=0.
\end{equation}
The main result of \cite{McCann97} states that under the
assumption
\begin{subequations}
\begin{equation}\label{mc}
P'(r)r\ge (1-{1}/{d})P(r)\geq 0, \quad \forall r\in (0,+\infty),
\end{equation}
or, equivalently,
\begin{equation*}
\label{eq:85} r\mapsto \frac{P(r)}{r^{1-1/d}}\quad \text{is
nonnegative and nondecreasing on }(0,+\infty),
\end{equation*}
\end{subequations}
the functional $\UU$ is convex along the constant speed geodesics
induced by $W$, i.e. for every curve $(\mu_s)_{s\in [0,1]} \subset
\Probac\Omega$ satisfying
\begin{equation}
\label{eq:2} W(\mu_{s_1},\mu_{s_2}) = |s_1-s_2|\, W(\mu_0,\mu_1)
\qquad \forall s_1,s_2 \in [0,1],
\end{equation}
the map $s\mapsto \UU(\mu_s)$ is convex in $[0,1]$. This class of
curves can be, equivalently, defined by displacement
interpolation, using the Brenier's optimal transportation map
pushing $\mu_0$ onto $\mu_1$ (see \cite{Vil03}, for example). For
power-like functions $U,P$
\begin{equation}
\label{eq:5} U(\rho)=
\begin{cases}
\frac 1{\beta-1}\rho^\beta&\text{if }\beta\neq 1,\\
\rho\log \rho&\text{if }\beta=1,
\end{cases}
\quad P(\rho)=\rho^\beta, \quad \text{\eqref{mc} is equivalent to
$\beta\ge 1-1/d$.}
\end{equation}

\paragraph{The link with a nonlinear diffusion equation.}
Among the various applications of this property, a remarkable one
concerns a wide class of nonlinear diffusion equations. The
seminal work of Otto \cite{Ott01} contributed the key idea that
a solution of the nonlinear diffusion equation
\begin{equation}\label{PMEin}
\de_t\rho-\nabla\cdot(\rho\nabla U'(\rho)) = 0\quad\text{in }
(0,+\infty)\times \Omega,
\end{equation}
with homogeneous Neumann boundary condition on $\partial\Omega$
can be interpreted as the trajectory of the \emph{gradient flow}
of $\UU$ with respect to the Wasserstein distance. This means that
the equation is formally the gradient flow of $\UU$ with respect
to the local metric which for a tangent vector $s$ has the form
\[ \langle s,s \rangle_\rho = \int_\Omega \rho |\nabla p|^2 dx \quad
\te{ where }\; \left\{
\begin{array}{rl}
- \nabla \cdot (\rho \nabla p) & = s \quad \te{in } \Omega \\
\nabla p \cdot \nn & = 0 \quad \te{on } \partial \Omega
\end{array} \right. \]
where $\nn$ is a unit normal vector to $\partial \Omega$. Let us
note here that the equation \eqref{PMEin} corresponds via
\eqref{eq:10} to
\begin{equation}
\label{eq:57} \de_t\rho-\Delta P(\rho)=0.
\end{equation}
In particular, the heat equation, for $P(\rho)=\rho$, is the
gradient flow of the logarithmic entropy $\UU(\rho)=\int_\Omega
\rho\log\rho\,\dx$. Let us also note that the metric above
satisfies
\[ \langle s,s \rangle_\rho = \inf \left\{\int_\Omega \rho |\vv|^2 dx \::\;
s + \nabla \cdot (\rho \vv) = 0 \te{ in } \Omega \te{ and } \vv \cdot
\nn =0 \te{ on } \partial \Omega \right\}. \]
The key property of
this metric is that the length of the minimal geodesic between
given two measures is nothing but the Wasserstein distance. More
precisely
\begin{equation*}
\label{eq:3}
\begin{aligned}
W^2(\mu_0,\mu_1) = \inf \Big\{\int_0^1\int_\Rd
|\vv_s(x)|^2\rho_s(x)\,\dx \, \d s :&\quad \de_s\rho + \nabla
\cdot (\rho\vv )=0
\ \mbox{in } (0,1)\times\Rd,\\
\supp(\rho_s)\subset \overline\Omega,& \quad \rho_0 \Leb{d}=\mu_0,
\ \rho_1 \Leb{d}=\mu_1 \Big\}.
\end{aligned}
\end{equation*}
This \emph{dynamical formulation} of the Wasserstein distance was
rigorously established by Benamou and Brenier in \cite{BB} and extended to more general
situations in \cite{AGS} and \cite{L07}.

As for the classical gradient flows of convex functions in
euclidean spaces, the flow associated with \eqref{PMEin} is a
contraction with respect to the Wasserstein distance. In
\cite{AGS} the authors showed that one of the possible ways to
rigorously express the link between the functional $\UU$, the
distance $W$, and the solution of the diffusion equation
\eqref{PMEin} is given by the \emph{evolution variational
inequality} satisfied by the measures $\mu_t=\rho(t,\cdot)\Leb d$
associated with \eqref{PMEin}:
\begin{equation}
\label{eq:1} \frac 12\frac{\d^+}{\d t}W^2(\mu_t,\nu)\le
\UU(\nu)-\UU(\mu_t)\quad \forall\, \nu\in \Probac\Omega.
\end{equation}

\paragraph{A new class of ``dynamical'' distances.}
In a number of problems from mathematical biology
\cite{H03a,BFD06,BDi,DiR}, mathematical physics
\cite{bib:k3,bib:k,bib:f1,bib:f2,SSC06,CLR08,CRS08}, studies of
phase segregation \cite{GL1, Sl08}, and studies of thin liquid
films \cite{Ber98}, the mobility of ``particles'' depends on the
density $\rho$ itself. More precisely the local metric in the
configuration space is formally given as follows: For a tangent
vector $s$ (euclidean variation)
\[ \langle s,s \rangle_\rho = \int_\Omega \Mob(\rho) |\nabla p|^2 dx \quad
\te{ where }\; \left\{
\begin{array}{rl}
- \nabla \cdot (\Mob(\rho) \nabla p) & = s \quad \te{in } \Omega \\
\nabla p \cdot \nn & = 0 \quad \te{on } \partial \Omega
\end{array} \right. \]
where $\Mob:[0,+\infty)\to [0,+\infty)$ is the \emph{mobility
function}. The global distance generated by the local metric is
given by
\begin{equation}\label{i:defwm}
\begin{aligned}
\Wmod_{\Mob,\Omega}^2&(\mu_0,\mu_1) :=\inf \Big\{\int_0^1\int_\Rd
|\vv_s(x)|^2\,\Mob(\rho_s(x))\,\dx \,\d s :
\\&\
\de_s\rho + \nabla \cdot \big(\Mob(\rho)\vv \big)=0 \ \mbox{in }
(0,1)\times\Rd,\ \supp(\rho_s)\subset \overline\Omega, \ \rho_0
\Leb{d}=\mu_0, \ \rho_1 \Leb{d}=\mu_1 \Big\}.
\end{aligned}
\end{equation}
This distance was recently introduced and studied in \cite{DNS} in
the case when $\Mob$ is concave and nondecreasing. Similarly to
the case $\Mob(r)=r$, it is easy to check formally that the
trajectory of the \emph{gradient flow} of $\UU$ with respect to
the modified distance $\Wmod_{\Mob,\Omega}$ solves
\begin{equation}\label{GFm}
\de_t\rho-\nabla\cdot (\Mob(\rho) \nabla U'(\rho)) = 0
\quad\text{in } (0,+\infty)\times \Omega
\end{equation}
with homogeneous Neumann boundary conditions on $\partial\Omega$.
Assuming that $U''\Mob$ and $U''\Mob\Mob'$ are locally integrable,
we can define in this case the function $P$ and the auxiliary
function $H$ by
\begin{equation*}
\label{eq:11} P(r):=\int_0^r U''(z)\,\Mob(z)\,\d z,\quad
H(r):=\int_0^r U''(z)\Mob(z)\Mob'(z)\,\d z= \int_0^r
P'(z)\,\Mob'(z)\,\d z,
\end{equation*}
so that
\begin{equation*}
\label{eq:12} P'=\Mob\,U'',\quad H'=\Mob'\,P',\quad P(0)=H(0)=0,
\end{equation*}
and, at least for smooth solutions, the problem \eqref{GFm} is
equivalent to \eqref{eq:57}.

By means of a formal computation, detailed in Section 2, the
second derivative of the internal energy functional $\UU$ along a
geodesic curve $(\mu_s)_{s\in [0,1]}$ satisfying as in
\eqref{eq:2}
\begin{equation*}
\label{eq:4} \Wmod_{\Mob,\Omega}(\mu_{s_1},\mu_{s_2})=|s_1-s_2|
\,\Wmod_{\Mob,\Omega}(\mu_0,\mu_1)\quad \forall\, s_1,s_2\in
[0,1],
\end{equation*}
is nonnegative, i.e. $\frac{\d^2}{\d s^2}\UU(\mu_s)\geq 0$, if the
following generalization of McCann condition (\ref{mc},b) holds
\begin{subequations}
\begin{equation}\label{gmc}
P'(r)\Mob(r)\ge(1-{1}/{d})H(r)\geq 0, \quad \forall r\in
(0,+\infty).
\end{equation}
It can also be expressed by requiring that
\begin{equation*}
\label{eq:86} r\mapsto \frac{H(r)}{\Mob^{1-1/d}(r)}=\frac
1{\Mob^{1-1/d}(r)} \int_0^r P'(s)\Mob'(s)\,\d s \quad \text{is
nondecreasing in }(0,+\infty).
\end{equation*}
\end{subequations}
As in the case of the Wasserstein distance, in dimension $d=1$ the
condition \eqref{gmc} reduces to the usual convexity of $U$. In
dimension $d\ge2$, still considering the relevant example of
power-like functions $U,P,\Mob$ as in \eqref{eq:5}, we get
\begin{equation*}
\label{eq:6} U(\rho)=
\begin{cases}
\frac 1{\beta-1}\rho^\beta&\text{if }\beta\neq 1\\
\rho\log \rho&\text{if }\beta=1
\end{cases},
\quad \Mob(\rho)=\rho^\alpha,\quad
P(\rho)=\frac\beta\gamma\rho^\gamma,\quad \gamma:=\alpha+\beta-1
\end{equation*}
and condition \eqref{gmc} is equivalent to
\begin{equation*}
\label{eq:7} \alpha\in (0,1],\quad \gamma\ge 1-\alpha/d.
\end{equation*}
In this case the heat equation corresponds to
$\gamma=\alpha+\beta-1=1$ and it is therefore the gradient flow of
the functional
\begin{equation*}
\label{eq:8} \UU(\rho)=\frac1{(2-\alpha)(1-\alpha)}\int_\Omega
\rho^{2-\alpha}\dx
\end{equation*}
with respect to the distance $\Wmod_{\Mob,\Omega}$ induced by the
mobility function $\Mob(\rho)=\rho^\alpha$.

Another interesting example, still leading to the heat equation,
is represented by the functional
\begin{equation*}
\label{eq:9} \UU(\rho)=\int_\Omega
\Big(\rho\log\rho+(1-\rho)\log(1-\rho)\Big)\,\dx, \quad 0\le
\rho\le 1\quad\text{$\Leb d$-a.e.\ in }\Omega,
\end{equation*}
and the distance induced by $\Mob(\rho)=\rho(1-\rho)$, $\rho\in
[0,1]$. Notice that in this case the positivity domain of the
mobility $\Mob$ is the finite interval $[0,1]$, a case that has
not been explicitly considered in \cite{DNS}, but that can be
still covered by a careful analysis (see \cite{LM}).

\paragraph{Geodesic convexity and contraction properties.}
Our aim is to prove rigorously the geodesic convexity of the
integral functional $\UU$ under conditions (\ref{gmc},b) and the
metric characterization of the nonlinear diffusion equation
\eqref{GFm} as the gradient flow of $\UU$ with respect to the
distance $\Wmod_{\Mob,\Omega}$ \eqref{i:defwm}. If one tries to
follow the same strategy which has been developed in the more
familiar Wasserstein framework, one immediately finds a serious
technical difficulty, due to the lackness of an ``explicit''
representation of the geodesics for $\Wmod_{\Mob,\Omega}$. In
fact, the McCann's proof of the displacement convexity of the
functionals $\UU$ is strictly related to the canonical
representation of the Wasserstein geodesics in terms of optimal
transport maps.

Existence of a minimal geodesic connecting two measures at a
finite $\Wmod_{\Mob,\Omega}$ distance has been proved by
\cite{DNS}. However, an explicit representation is no longer
available. On the other hand in \cite{DS}, following the eulerian
approach introduced in \cite{OttoWest}, the authors presented a
new proof of McCann's convexity result for integral functionals
defined on a compact manifold without the use of the
representation of geodesics. Here, following the same approach of
\cite{DS}, we reverse the usual strategy which derives the
existence and the contraction property of the gradient flow of a
functional from its geodesic convexity. On the contrary, we show
that under the assumption \eqref{gmc} smooth solutions of
\eqref{GFm} satisfy the following Evolution Variational
Inequality analogous to \eqref{eq:1}
\begin{equation}\label{EVIin}
\frac12\frac{\d^+}{\dt }\Wmod_{\Mob,\Omega}^2(\mu_t,\nu)\leq
\UU(\nu)-\UU(\mu_t), \qquad \forall t\in [0,+\infty),\ \forall
\nu\in \Prob\Omega : \Wmod_\Mob(\nu,\mu_0)<+\infty.
\end{equation}
This is sufficient to construct a nice gradient flow generated by
$\UU$ and metrically characterized by \eqref{EVIin}, as showed in
\cite{AGS} and \cite{AShand}. The remarkable fact proved by
\cite{DS} is that whenever a functional $\UU$ admits a flow,
defined at least in a dense subset of $D(\UU)$, satisfying
\eqref{EVIin}, the functional itself is convex along the geodesics
induced by the distance $\Wmod_{\Mob,\Omega}$. As a by-product we
obtain stability, uniqueness, and regularization results for the
solutions of the problem \eqref{GFm} in a suitable subspace of
$\Prob\Omega$ metrized by $\Wmod_{\Mob,\Omega}$.

Concerning the assumptions on $\Mob$, its \emph{concavity} is a
necessary and sufficient condition to write the definition of
$\Wmod_{\Mob,\Omega}$ with a jointly convex integrand \cite{DNS}, which is
crucial in many properties of the distance, in particular for its
lower semicontinuity with respect to the usual weak convergence of
measures. Since $\Mob \geq 0$ on $[0, \infty)$ the concavity
implies that the mobility must be nondecreasing. This is the case
considered in \cite{DNS}. However we are also able to treat the
case when the mobility is defined on an interval $[0,M)$ where it
is nonnegative and concave. It that case the configuration space
is restricted to absolutely continuous measures with densities
bounded from above by $M$. Such mobilities are of particular
interest in applications as mentioned before.

\paragraph{Plan of the paper.}
In next section, we show the heuristic computations for the
convexity of functionals with respect to $\Wmod_{\Mob,\Omega}$. Section 3
is devoted to introduce the notation and to review the needed
concepts on $\Wmod_{\Mob,\Omega}$ from \cite{DNS}. Moreover, we
prove a key technical regularization lemma: Lemma
\ref{le:smoothing}. Subsection \ref{subsec:finite} addresses the
question of finiteness of $\Wmod_{\Mob,\Omega}(\mu_0,\mu_1)$,
providing new sufficient conditions on $\Mob$ and $\mu_0,\mu_1$ in
order to ensure that $\Wmod_{\Mob,\Omega}(\mu_0,\mu_1)<+\infty$.
After a brief review of some basic properties of the diffusion
equation \eqref{eq:57}, in Section 4 we try to get some insight on
the features of the generalized McCann condition (\ref{gmc},b), we
recall some basic facts on the metric characterization of
contracting gradient flows and their relationships with geodesic
convexity borrowed from \cite{AGS,DS}, and we state our main
results Theorems \ref{th:main} and \ref{th:conv}. The core of our
argument in smooth settings is collected in Section
\ref{sec:smooth}, whereas the last Section concludes the proofs of
the main results. At the end of the paper we collect some final
remarks and open problems.

\section{Heuristics}

We first discuss, in a formal way, the conditions for the
displacement convexity of the internal, the potential and the
interaction energy, with respect to the geodesics corresponding to
the distance \eqref{i:defwm}. For simplicity, we assume that
$\Omega=\R^d$ and that densities are smooth and decaying fast
enough at infinity so that all computations are justified.

\subsection{Geodesics}
We first obtain the optimality condition for the geodesic
equations in the fluid dynamical formulation of the the new
distance \eqref{i:defwm}. As in \cite{B03}, we insert the
nonlinear mobility continuity equation \eqref{i:defwm}
\begin{equation}\label{geo1}
\de_s\rho + \nabla \cdot (\Mob(\rho)\vv )=0 \qquad \mbox{in }
(0,1)\times\R^d.
\end{equation}
inside the minimization problem as a Lagrange multiplier. As a
result, we get the unconstrained minimization problem
\begin{align*}
\Wmod_\Mob^2(\mu_0,\mu_1) = \inf_{(\rho,\svv)} \sup_\psi \Big\{&
\int_0^1\int_\Rd \frac12 |\vv_s(x)|^2\Mob(\rho_s(x))\,\dx \,
\d s\\
&- \int_0^1\int_\Rd \left[\rho_s(x) \de_s \psi(s,x)
+\Mob(\rho_s(x)) (\vv_s(x)\cdot \nabla \psi(s,x))\right] \,\dx \,
\d s \\
&+ \int_\Rd \rho_1(x) \psi(1,x) \,\dx - \int_\Rd \rho_0(x)
\psi(0,x) \,\dx \Big\}.
\end{align*}
Applying a formal minimax principle and thus taking first an
infimum with respect to $\vv$ we obtain the optimality condition
$\vv=\nabla \psi,$ and the following formal characterization of
the distance
\begin{align*}
\Wmod_\Mob^2(\mu_0,\mu_1) = \sup_\psi \inf_{\rho}
\Big\{&-\frac12\int_0^1\int_\Rd |\nabla \psi |^2\Mob(\rho)\,\dx \,
\d s - \int_0^1\int_\Rd \rho \de_s \psi \,\dx \, \d s
\\
&+ \int_\Rd \rho_1(x) \psi(1,x) \,\dx - \int_\Rd \rho_0(x)
\psi(0,x) \,\dx \Big\},
\end{align*}
which provides the further optimality condition
\begin{equation}\label{geo2}
\de_s \psi + \frac12 \Mob'(\rho_s(x)) |\nabla \psi|^2 = 0 .
\end{equation}
We thus end up with a coupled system of differential equations in
$(0,1)\times \Rd$ \cite[Rem. 5.19]{DNS}
\begin{equation}
\label{eq:13} \left\{
\begin{aligned}
\de_s\rho + \nabla \cdot (\Mob(\rho)\nabla\psi )&=0,\\
\de_s \psi + \frac12 \Mob'(\rho) |\nabla \psi|^2 &= 0.
\end{aligned}
\right.
\end{equation}
\subsection{Internal energy}
We use the formal equations \eqref{eq:13} for the geodesics
associated to the distance \eqref{i:defwm} to compute the
conditions under which the internal energy functional is
displacement convex. If therefore $(\rho_s,\psi_s)$ is a smooth
solution of \eqref{eq:13}, which decays sufficiently at infinity,
we proceed as usual \cite{CMV03,Vil03,OttoWest} to obtain the
following formulas:
$$
\frac\d{\d s} \UU(\rho) =- \int_\Rd P(\rho) \Delta \psi \,\dx,
$$
and
\begin{align*}
\frac{\d^2}{\d s^2} \UU(\rho) =& \int_\Rd (P'(\rho)\Mob(\rho)-H(\rho))(\Delta \psi)^2\,\dx \\
& + \int_\Rd H(\rho)(-\nabla\psi\cdot\nabla \Delta\psi+\frac{1}{2}\Delta |\nabla\psi|^2)\,\dx \\
& - \frac12\int_\Rd P'(\rho)\Mob''(\rho) |\nabla \rho|^2
|\nabla\psi|^2\,\dx .
\end{align*}
As usual, the Bochner formula
$$
-\nabla\psi\cdot\nabla \Delta\psi+\frac{1}{2}\Delta |\nabla\psi|^2
= |\Hess \psi|^2 \geq \frac1d (\Delta \psi)^2,
$$
and the fact that $H(\rho)\geq 0$, allow us to estimate it as
\begin{align*}
\frac{\d^2}{\d s^2} \UU(\rho) \ge \int_\Rd
(P'(\rho)\Mob(\rho)-(1-{1}/{d}) H(\rho))(\Delta \psi)^2\,\dx
-\frac12 \int_\Rd P'(\rho)\Mob''(\rho) |\nabla \rho|^2
|\nabla\psi|^2\,\dx .
\end{align*}
Therefore, under conditions of concavity of the mobility
$\Mob(\rho)$ and the generalized displacement McCann's condition
\eqref{gmc}, the functional $\UU$ is convex along the geodesics of
the distance $\Wmod_\Mob$.

\subsection{Potential energy}
Similar heuristic formulas can be obtained for the potential and
the interaction energy, as in \cite{CMV03,Vil03}. We consider the
potential energy functional
\begin{equation*}
\VV(\mu):= \int_\Rd V(x) \,\d\mu,
\end{equation*}
with $V$ a given smooth potential. As before, it is easy to check
that the second derivative of $\VV$ along a geodesic satisfying
\eqref{eq:13} is
\begin{align*}
\frac{\d^2}{\d s^2} \VV(\rho) =& \int_\Rd \Mob(\rho) \Mob'(\rho)
\,(\Hess V\,\nabla \psi)
\cdot \nabla\psi\,\dx \\
& + \int_\Rd \Mob(\rho)\Mob''(\rho) \Big((\nabla \rho\cdot
\nabla\psi)(\nabla V\cdot \nabla\psi)\, - \frac12 (\nabla
\rho\cdot \nabla V) |\nabla\psi|^2\Big)\,\dx .
\end{align*}

This formula allows us to show that this functional cannot be
convex along geodesics if $\Mob$ is not linear. Technically, the
reason is the presence of the terms linearly depending on $\nabla
\rho$. We present a simple example:

\

\textbf{Example.} Let us first construct the example in one
dimension. The expression for the second derivative of the
functional above reduces to
\begin{equation*} \label{Vss}
\frac{d^2}{\d s^2} \VV(\rho) = \int_\R \Mob(\rho) \Mob'(\rho)
V_{xx} \, \psi_x^2 \, \dx + \frac{1}{2} \int_\R \Mob(\rho)
\Mob''(\rho) \rho_x \, V_x\, \psi_x^2 \, \dx =: I + II
\end{equation*}
Consider the case that $V$ is nontrivial. Then $V_x \neq 0$ on
some interval. For notational simplicity, we assume that
\[V_x > 0 \te{ on } [-2,2]. \]
Since the mobility $\Mob$ we are considering is not a linear
function of $\rho$ there exists $z>0$ such that $\Mob''(z) \neq
0$. Again for notational simplicity, let us assume that
\[ \Mob''(z) < 0 \te{ on } \left[\frac{1}{2}, \frac{3}{2} \right]. \]
The fact that we chose $V_x$ to be positive and $\Mob''$ negative
is irrelevant because the sign of term $II$ can be controlled by
the sign of $\rho_x$. Let $\eta$ be a piecewise linear function on
$\R$:
\begin{equation*} \label{eta}
\eta(x) = \left\{ \begin{array}{ll}
\frac{3}{2} \quad & \te{if } x< - \frac{1}{2} \\
1-x & \te{if } x \in \left[ -\frac{1}{2}, \frac{1}{2} \right] \\
\frac{1}{2} & \te{if } x > \frac{1}{2}.
\end{array} \right.
\end{equation*}
The fact that the function is Lipschitz, but not smooth is
irrelevant; smooth approximations of the given $\eta$, can also be
used in the construction. Let $\eta_\veps(x)= \eta
\left(\frac{x}{\veps}\right)$. Let $\sigma \in C^\infty_0(\R,
[0,1])$, supported in $[-1,1]$, such that $\sigma =1$ on
$\left[-\frac{1}{4}, \frac{1}{4} \right]$ and $\int_\R \sigma(x)
\dx =1$. Let $\rho_\veps = \sigma \eta_\veps$. Note that $\int_\R
\rho_\veps \dx =1$. A typical profile of $\rho_\veps$ is given in
Figure \ref{figcounter}.

\begin{figure}[ht]
\centering

\resizebox{5.25in}{!}{\includegraphics{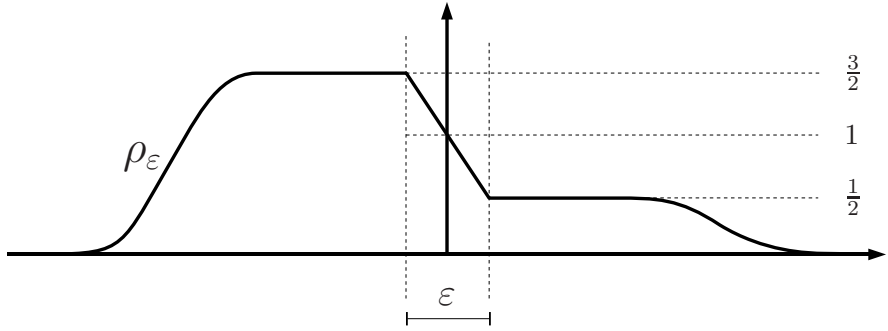}}
\put(-322,70){{\LARGE $\rho_\veps$}} \put(-204,16){{\Large
$\veps$}} \put(-52,54){$\frac{1}{2}$} \put(-52,77){$1$}
\put(-52,101){$\frac{3}{2}$}

\caption{A profile at which the potential energy is not convex.}
\label{figcounter}
\end{figure}

The test velocity (tangent vector at $s=0$) we consider also needs
to be localized near zero. A simple choice is $\psi_\veps(0) =
\eta_\veps$. Let $\rho_\veps(s)$ be the corresponding geodesics
given by \eqref{geo1} and \eqref{geo2}. Let us observe how, at
$s=0$, the terms $I$ and $II$ scale with $\veps$:
\begin{align*}
I_\veps & \leq \max_{z \in [0,2]} \Mob(z) \Mob'(z) \, \max_{x \in
[-1,1]}
V_{xx}(x)\, \frac{1}{\veps^2} \, \veps \sim \frac{1}{\veps} ,\\
II_\veps & \leq - \frac{1}{2} \min_{z \in [\frac{1}{2},
\frac{3}{2}]} \Mob(z) |\Mob''(z)| \, \frac{1}{\veps}\, \min_{x \in
[-1,1]} V_x \, \frac{1}{\veps^2} \, \veps \sim - \frac{1}{\veps^2} .
\end{align*}
Thus, for $\veps$ small enough, $\left.\frac{d^2}{\d
s^2}\right|_{s=0} \VV(\rho_\veps(s)) < 0$. Furthermore note that
the square of the length of the tangent vector $\frac{d}{\dt }
\rho_\veps(0)$ is
\[ \int_\R \Mob(\rho_\veps(0)) \, (\partial_x \psi_\veps)^2 \, \dx \sim \frac{1}{\veps} \]
Thus for any $\lambda \in \R$ there exists $\veps>0$ such that
\[ \left.\frac{d^2}{\d s^2}\right|_{s=0} \VV(\rho_\veps(s)) +
\lambda \int_\R \Mob(\rho_\veps(0)) \, (\partial_x \psi_\veps)^2
\, \dx < 0 \] which implies that $\VV$ is not $\lambda$-convex for
any $\lambda \in \R$.

Let us conclude the example by remarking that it can be extended
to multidimensional domains. In particular it suffices to extend
the 1-D profile to d-D to be constant in every other direction and
then use a cut-off. We only sketch the elements of the
construction.

We can assume that $\nabla V(0) = e_d$. Let $\tilde \rho_\veps(x)
= \rho_\veps(x_d)$. Let $\hat x = (x_1, \dots, x_{d-1})$. To
cut-off in the directions perpendicular to $e_d$ we use the length
scales $1 \gg l \gg \delta \gg \veps$. Let $\theta_{l,\delta}$ be
smooth cut-off function equal to 1 on $[-l,l]$ and equal to $0$
outside of $[-l-\delta, l+\delta]$; with $|\nabla
\theta_{l,\delta}| < \frac{C}{\delta}$ and $|D^2
\theta_{l,\delta}| < \frac{C}{\delta^2}$. Let $\rho_{l, \delta,
\veps}(x) = \tilde \rho_\veps(x_d) \theta_{l,\delta}(|\hat x|)$.
Let $\psi_{l, \delta, \veps}(x)= \eta_\veps(x_d) \theta_{l,
\delta}(|\hat x|)$. Checking the scaling of appropriate terms is
straightforward.

\subsection{Interaction energy}

Consider the interaction energy functional
\begin{equation*}
\WW(\rho):= \frac12 \int_{\R^d} \int_{\R^d} W(x-y) \rho(x)\,
\rho(y)\, dx\,dy ,
\end{equation*}
with $W$ a given smooth potential. As before, it is easy to check
that
\begin{align*}
\frac{d^2}{ds^2} \WW(\rho) =& \int_{\R^d}\int_{\R^d} \Mob(\rho(x))
\,\Mob'(\rho(x))\, \rho(y) \,\nabla \psi(x) \cdot (\Hess W(x-y) \nabla \psi(x))\,dx\, dy \\
& - \int_{\R^d}\int_{\R^d} \Mob(\rho(x)) \,\Mob(\rho(y)) \, \nabla
\psi(y) \cdot
(\Hess W(x-y) \nabla \psi(x) )\,dx\, dy \\
& + \int_{\R^d}\int_{\R^d}
\Mob(\rho(x))\,\Mob''(\rho(x))\,\rho(y)\, (\nabla \rho(x)\cdot
\nabla \psi(x))\,(\nabla W(x-y)\cdot
\nabla \psi(x))\,dx \,dy \\
& - \frac12 \int_{\R^d}\int_{\R^d}
\Mob(\rho(x))\,\Mob''(\rho(x))\,\rho(y)\, (\nabla \rho(x) \cdot
\nabla W(x-y))\, |\nabla\psi(x)|^2\,dx\,dy.
\end{align*}
It can be demonstrated that if $\Mob$ is non-linear then the
interaction energy is not geodesically convex. As for the
potential energy, the reason lies in the presence of derivatives
of $\rho$ in the expression above. More precisely, in one
dimension the second derivative of $\WW(\rho)$ reduces to
\begin{align*}
\frac{d^2}{ds^2} \WW(\rho) =& \int_\R\int_\R \Mob(\rho(x))
\,\Mob'(\rho(x))\, \rho(y) \,\psi_x^2(x) W_{xx}(x-y) \,dx\, dy \\
& - \int_\R\int_\R \Mob(\rho(x)) \,\Mob(\rho(y)) \, \psi_y(y)
W_{xx}(x-y) \psi_x(x) )\,dx\, dy \\
& + \frac12 \int_\R\int_\R
\Mob(\rho(x))\,\Mob''(\rho(x))\,\rho(y)\,
\rho_x(x)\psi_x^2(x)\,W_x(x-y)\,dx \,dy .
\end{align*}
It turns out that the example for the lack of (semi-)convexity
provided for the potential energy is also an example (with $V$
replaced by $W$) for the interaction energy. The estimates of the
terms are similar, so we leave the details to the reader.


\section{Notation and preliminaries}
In this section, following \cite{DNS}, we shall recall the main
properties of the distance $\Wmod_{\Mob,\Omega}$ introduced in
\eqref{i:defwm}. For the sake of simplicity, we only consider here
the case of a bounded open domain $\Omega$, so that it is not be
restrictive to assume that all the measures (Radon, i.e.\ locally
finite, in the general approach of \cite{DNS}) involved in the
various definitions have finite total variation. Since we deal
with arbitrary mobility functions $\Mob$, these distances do not
exhibit nice homogeneity properties as in the Wasserstein case;
therefore we deal with finite Borel measures without assuming that
their total mass is $1$.

\subsection{Measures and continuity equation}
We denote by $\MM^+(\R^d)$ (resp.~$\MM^+_\comp(\Rd)$) the space of
finite positive Borel measures on $\R^d$ (resp.~with compact
support) and by $\MM(\R^d;\R^d)$
the space of $\R^d$-valued Borel measures on $\R^d$ with finite
total variation. By Riesz representation theorem, the space
$\MM(\R^d;\R^d)$ can be identified with the dual space of
$C^0_\comp(\R^d;\R^d)$ and it is endowed with the corresponding
weak$^*$ topology. We denote by $|\nnu|\in \MM^+(\Rd)$ the total
variation of the vector measure $\nnu\in \MM(\R^d;\R^d)$. $\nnu$
admits the polar decomposition $\nnu = \ww |\nnu|$ with $\ww\in
L^1(|\nnu|;\R^d)$. If $B$ is a Borel subset of $\Rd$ (typically
an open or closed set) we denote by $\MM^+(B)$
(resp.~$\MM^+(B;\Rd)$) the subset of $\MM^+(\Rd)$ (resp.
$\MM(\Rd;\Rd)$) whose measure $\mu$ are concentrated on $B$, i.e.\
$\mu(\Rd\setminus B)=0$ (resp.~$|\mu|(\Rd\setminus B)=0$). Notice
that if $B$ is a compact subset of $\Rd$ then the convex set in
$\MM^+(B)$ of measures with a fixed total mass $\mass$ is compact
with respect to the weak$^*$ topology. If $\mass>0$,
$\MM^+(B,\mass)$ is the convex subset of $\MM^+(B)$ whose measures
have fixed total mass $\mu(B)=\mass$.

Let $\Omega$ be a bounded open subset of $\Rd$. Given
$\mu_0,\mu_1\in \MM^+(\overline\Omega)$ we denote by
$\CE_\Omega(\mu_0\to\mu_1)$ the collection of time dependent
measures $(\mu_s)_{s\in [0,1]}\subset \MM^+(\overline\Omega)$ and
$(\nnu_s)_{s\in (0,1)}\in \MM(\overline\Omega;\Rd)$ such that
\begin{enumerate}
\item $s\mapsto \mu_s$ is weakly$^*$ continuous in $\MM^+(\Rd)$
with $\mu|_{s=0}=\mu_0$ and $\mu|_{s=1}=\mu_1$;

\item $(\nnu_s)_{s\in (0,1)}$ is a Borel family with $\int_0^1
|\nnu_s|(\overline\Omega)\,\d s<+\infty$;

\item $(\mu,\nnu)$ is a distributional solution of
\begin{equation*}
\label{eq:14} \de_s \mu_s +\nabla \cdot \nnu_s =0\quad\text{in
}(0,1)\times \Rd.
\end{equation*}
\end{enumerate}
If $(\mu,\nnu)\in \CE_\Omega(\mu_0\to\mu_1)$ then it is immediate to
check that the total mass $\mu_s(\Rd)=
\mu_s(\overline\Omega)=\mass$ is a constant, independent of $s$.
In particular, $\mu_0(\Rd)=\mu_1(\Rd)$.

\subsection{Mobility and action functional}
We fix a right threshold $M\in (0,+\infty]$ and a concave
\emph{mobility function} $\Mob\in C^0[0,M)$ strictly positive in
$(0,M)$. We denote by $\Mob(M)$ the left limit of $\Mob(r)$ as
$r\uparrow M$. We can also introduce the maximal left interval of
monotonicity of $\Mob$ whose right extreme is
\begin{equation*}
\label{eq:61} \med:=\sup\Big\{m\in [0,M):\Mob\restr{[0,m]}\quad
\text{is nondecreasing}\Big\}.
\end{equation*}
We distinguish two situation:
\begin{description}
\item[\underline{Case A}] $M=+\infty$
so that $\Mob$ is nondecreasing and $\med=M=+\infty$; typically
$\Mob(0)=0$ and the main example is provided by
$\Mob(r)=r^\alpha$, $\alpha\in [0,1]$. This is the case considered
in \cite{DNS}. When $\Mob'(+\infty):=
\lim_{r\uparrow+\infty}r^{-1}\Mob(r)=\lim_{r\uparrow+\infty}\Mob'(r)=0$
we are in the \emph{sublinear growth} case. A \emph{linear growth}
of $\Mob$ corresponds to $\Mob'(+\infty)>0$.

\item[\underline{Case B}] $M<+\infty$, so that $0\le \med\le M$
and $\Mob$ is nonincreasing in the right interval $[\med,M]$ (but
we also allow $\Mob$ to be constant or even decreasing in $[0,M)$
with $\med=0$). Typically $\Mob(0)=\Mob(M)=0$ (in this case
$0<\med<M$) and the main example is $\Mob(r)=r(M-r)$, or, more
generally, $\Mob(r)=r^{\alpha_0}(M-r)^{\alpha_1}$,
$\alpha_0,\alpha_1\in (0,1]$.
\end{description}

Many properties proved in the case A can be extended to the case
B, but there are important exceptions: we refer to \cite{LM} for
further details. Using the conventions
\begin{equation}
\label{eq:58} \begin{array}{lc} a/b=0 & \text{if }a=b=0,\\
a/b=+\infty & \text{if }a>0=b,
\end{array}
\end{equation}
the corresponding \emph{action density function} $ \phi_\Mob:\R
\times \R^d\to [0,+\infty]$ is defined by
\begin{equation*}
\label{eq:17} \phi_\Mob(\rho,\ww)= \left\{\begin{array}{ll}
\displaystyle \frac{|\ww|^2}{\Mob(\rho)} &\text{if } \rho \in [0,M], \\
+\infty &\text{if } \rho <0\text{ or }\rho>M.
\end{array}\right.
\end{equation*}
It is not difficult to check that, under the convention
\eqref{eq:58}, the function $\phi_\Mob$ is (jointly) convex and
lower semi-continuous.

Given that $\Mob$ is concave and $\phi_\Mob$ is convex, when
$M=+\infty$ we can define the recession function
$\varphi^\infty_\Mob:\Rd\mapsto [0,+\infty]$ (recall
\eqref{eq:58})
\begin{equation*}
\label{eq:14bis} \varphi^\infty_\Mob(\ww):=
\lim_{r\uparrow+\infty}r\phi_\Mob(1,\ww/r)=
\frac{|\ww|^2}{\Mob'(\infty)} ,\quad \Mob'(\infty) :=
\lim_{r\to+\infty}\Mob'(r) =
\lim_{r\to+\infty}\frac{\Mob(r)}{r}\ge0.
\end{equation*}
We introduce now the \emph{action functional}
\begin{equation*}
\Phi_{\Mob,\Omega} :\MM^+(\R^d)\times \MM(\R^d;\R^d) \to
[0,+\infty],
\end{equation*}
defined on couples of measures $\mu\in \MM^+(\R^d)$, $\nnu\in
\MM(\R^d;\R^d)$. In order to define it we consider the usual
Lebesgue decomposition $\mu=\rho \Leb{d} + \mu^\perp $, $\nnu=\ww
\Leb{d}+ \nnu^\perp $ and distinguish the following cases:
\begin{enumerate}
\item If the support of $\mu$ or $\nnu$ is not contained in
$\overline\Omega$ then $\Phi_{\Mob,\Omega}(\mu,\nnu)=+\infty$;
\item When $M<+\infty$ (Case B), we set
\begin{equation*}
\label{eq:15} \Phi_{\Mob,\Omega}(\mu,\nnu):=
\begin{cases}
\displaystyle\int_\Omega \phi_\Mob(\rho,\ww)\,\dx
&\tif\mu^\perp=0,\ \nnu^\perp=0\\
+\infty&\text{otherwise};
\end{cases}
\end{equation*}
notice that if $\Phi_{\Mob,\Omega}(\mu,\nnu)<+\infty$ then
$\rho\in L^\infty(\Omega)$ with $0\le \rho\le M$ $\Leb d$-a.e.\ in
$\Omega$ and $\ww\in L^2(\Omega;\Rd)$. \item When $M=+\infty$ and
$\Mob'(\infty)=0$ (Case A, sublinear growth) then
\begin{equation*}
\label{eq:16} \Phi_{\Mob,\Omega}(\mu,\nnu):=
\begin{cases}
\displaystyle\int_\Omega \phi_\Mob(\rho,\ww)\,\dx
&\tif\nnu^\perp=0\\
+\infty&\text{otherwise};
\end{cases}
\end{equation*}
\item Finally, when $M=+\infty$ and $\Mob'(\infty)>0$ (Case A,
linear growth) then we set
\begin{equation*}
\label{eq:16bis} \Phi_{\Mob,\Omega}(\mu,\nnu):=
\begin{cases}
\displaystyle\int_\Omega \phi_\Mob(\rho,\ww)\,\dx+
\int_{\overline\Omega} \varphi^\infty_\Mob(\ww^\perp)\,\d\mu^\perp
&\tif\nnu^\perp=\ww^\perp\mu^\perp\ll\mu^\perp\\
+\infty&\text{otherwise}.
\end{cases}
\end{equation*}
\end{enumerate}

\subsection{The modified Wasserstein distance}
Let $\Omega$ be a bounded open set. Given $\mu^0,\mu^1 \in
\MM^+(\overline\Omega)$ we define
\begin{align}\label{def:Wm}
\Wmod_{\Mob,\Omega}(\mu^0,\mu^1):=&\inf\left\{ \Big(\int_0^1
\Phi_{\Mob,\Omega} (\mu_s,\nnu_s) \,\d s\Big)^{1/2} :
(\mu,\nnu)\in \CE_{\Omega}(\mu^0 \to\mu^1)\right\}
\\=&
\label{eq:70} \inf\left\{\int_0^1 \Big(\Phi_{\Mob,\Omega}
(\mu_s,\nnu_s)\Big)^{1/2} \,\d s : (\mu,\nnu)\in \CE_\Omega(\mu^0
\to\mu^1)\right\}.
\end{align}
We refer to \cite[Thm. 5.4]{DNS} for the equivalence between
\eqref{def:Wm} and \eqref{eq:70}.
$\Wmod_{\Mob,\Omega}(\mu^0,\mu^1)=+\infty$ if the set of
connecting curves $\CE_{\Omega}(\mu^0\to\mu^1)$ is empty. The
following three propositions are proved in \cite{DNS}, see
Theorems 5.5, 5.6, 5.7, 5.15, and Proposition 5.14.

\begin{proposition}
The space $\MM^+(\overline\Omega)$ endowed with the distance
$\Wmod_{\Mob,\Omega}$ is a complete pseudo-metric space (the
distance can assume the value $+\infty$), inducing as strong as,
or stronger topology than the weak$^*$ one.

Given a measure $\sigma \in \MM^+(\overline\Omega)$, the space
$\MM^+_{\Mob,\Omega}[\sigma]:= \big\{\mu\in
\MM^+(\overline\Omega):
\Wmod_{\Mob,\Omega}(\mu,\sigma)<+\infty\big \}$ is a complete
metric space whose measures have the same total mass of $\sigma$.

Moreover, for every $\mu_0,\ \mu_1\in\MM^+(\overline\Omega)$ such
that $ \Wmod_{\Mob,\Omega}(\mu_0,\mu_1)<+\infty$ there exists a
minimizing couple $(\mu,\nnu)$ in \eqref{def:Wm} (unique, if
$\Mob$ is strictly concave and sublinear) and the curve
$(\mu_s)_{s\in[0,1]}$ is a constant speed geodesic for
$\Wmod_{\Mob,\Omega}$, thus satisfying
\begin{equation*}
\Wmod_{\Mob,\Omega}(\mu_t,\mu_s) =
|t-s|\Wmod_{\Mob,\Omega}(\mu_0,\mu_1) \qquad \forall s,t \in
[0,1].
\end{equation*}
\end{proposition}

\begin{proposition}[Lower semi-continuity]
\label{prop:semi} If $\Omega_n,\Omega$ are bounded open sets such
that $\Leb{d}\restr{\Omega_n}$ weakly* converges to
$\Leb{d}\restr{\Omega}$, $M_n\in (0,+\infty]$ is a nonincreasing
sequence converging to $M$, $\Mob_n$ is a sequence of nonnegative
concave functions in the intervals $(0,M_n)$ such that
\begin{displaymath}
\Mob_{n'}(r)\ge \Mob_{n''}(r)\quad\forall\, r\in (0,M_{n''})\quad
\text{if }n'\le n'',\qquad \lim_{n\to\infty}
\Mob_n(r)=\Mob(r)\quad \forall\, r\in (0,M),
\end{displaymath}
and $\mu_0^n$, $\mu_1^n$ are sequences of measures weakly*
convergent to $\mu_0$ and $\mu_1$ respectively, then
\begin{equation}\label{lsc}
\liminf_{n\to +\infty}\Wmod_{\Mob_n,\Omega_n}(\mu_0^n,\mu_1^n)\geq
\Wmod_{\Mob ,\Omega}(\mu_0,\mu_1) .
\end{equation}
\end{proposition}

\begin{proposition}[Monotonicity]
Let $\tilde\Omega \supset \Omega$, $\tilde\Mob\ge \Mob$ $\mu_0$,
$\mu_1\in\MM^+(\overline\Omega)$. Then the following inequality
holds
\begin{equation*}\label{mon1}
\Wmod_{\tilde\Mob ,\tilde\Omega}(\mu_0,\mu_1)\leq \Wmod_{\Mob
,\Omega}(\mu_0,\mu_1).
\end{equation*}
\end{proposition}

\begin{proposition}
\label{prop:convolution} Let $k\in C^\infty_c(\R^d)$ be a
nonnegative convolution kernel, with $\int_{\R^d}k(x)\,\dx =1$ and
$\supp(k)=\overline B_1(0)$, and let
$k_\eps(x):=\eps^{-d}k(x/\eps)$. For every $\mu,\mu_0,\mu_1\in
\MM^+(\overline\Omega)$ and $\nnu\in \MM(\overline\Omega;\Rd)$ we
have
\begin{align}
\Phi_{\Mob,\Omega_\eps}(\mu\ast k_\eps,\nnu\ast
k_\eps)&\le \Phi_{\Mob,\Omega}(\mu,\nnu)&& \forall\, \eps>0,
\nonumber
\\
\label{convusc} \Wmod_{\Mob ,\Omega_\eps}(\mu_0\ast k_\eps,
\mu_1\ast k_\eps )&\leq \Wmod_{\Mob ,\Omega}(\mu_0,\mu_1)
&&\forall\, \eps>0 ,
\\\label{convcont}
\lim_{\eps \to 0}\Wmod_{\Mob ,\Omega_\eps}(\mu_0\ast k_\eps ,
\mu_1\ast k_\eps ) &= \Wmod_{\Mob ,\Omega}(\mu_0,\mu_1) ,&&
\end{align}
where $\Omega_\eps:=\Omega+B_\eps(0)$.
\end{proposition}
\begin{proof}
If $\Phi_{\Mob,\Omega}(\mu,\nnu)<+\infty$ then $\mu,\nnu$ are
supported in $\overline\Omega$ and \cite[Theorem 2.3]{DNS} yields
\begin{displaymath}
\Phi_{\Mob,\Omega}(\mu,\nnu)= \Phi_{\Mob,\Rd}(\mu,\nnu)\ge
\Phi_{\Mob,\Rd}(\mu\ast k_\eps,\nnu\ast k_\eps)=
\Phi_{\Mob,\Omega_\eps}(\mu\ast k_\eps,\nnu\ast k_\eps),
\end{displaymath}
being $\mu\ast k_\eps,\nnu\ast k_\eps$ supported in $\overline\Omega_\eps$.
Notice that only the concavity of $\Mob$ (and not its
monotonicity) plays a role here. A similar argument and
\cite[Theorem 5.15]{DNS} yields \eqref{convusc}. The limit
\eqref{convcont} is an immediate consequence of \eqref{lsc} and \eqref{convusc}.
\end{proof}

The next technical lemma provides a crucial approximation result
for curves with finite $\Phi_{\Mob,\Omega}$ energy. It allows for
measures to be approximated by ones with smooth, positive
densities.

\begin{lemma}
\label{le:smoothing} Let $\Omega$ be an open bounded convex set
and let $(\mu,\nnu)\in \CE_{\Omega}(\mu_0\to\mu_1)$ with given
constant mass $\mass$ and finite energy $\int_0^1
\Phi_{\Mob,\Omega}(\mu_s,\nnu_s)\,\d s<+\infty$. For every
$\eps>0,\delta\in [0,1]$ there exist a decreasing family of smooth
convex sets $\Omega^\eps\downarrow\Omega$ and a family of curves
$(\mu^{\eps,\delta},\nnu^{\eps,\delta})\in \CE_{\Omega^\eps}
(\mu^{\eps,\delta}_0\to\mu^{\eps,\delta}_1)$ with the following
properties
\begin{gather}
\label{eq:87} \mu^{\eps,\delta}_i=(1-\delta)\mu_i\ast
k_\eps+\delta \lambda^\eps,\quad \lambda^\eps:=\frac\mass{\Leb
d(\Omega^\eps)}\Leb d\restr{\Omega^\eps}, \quad
\mu^{\eps,\delta}_s(\Omega^\eps)=\mass,\\
\label{eq:90} \mu^{\eps,\delta}_s=\rho^{\eps,\delta}_s\Leb
d\restr{\Omega^\eps},\quad
\nnu^{\eps,\delta}_s=\ww^{\eps,\delta}_s\Leb d\restr{\Omega^\eps},
\quad \rho^{\eps,\delta},\ww^{\eps,\delta}\in C^\infty([0,1]\times
\overline\Omega^\eps),
\\
\partial_s \rho^{\eps,\delta}_s+\nabla\cdot \ww^{\eps,\delta}_s=0\quad\text{in }
(0,1)\times \Omega^\eps,\quad \rho^{\eps,\delta}\ge \delta
\frac\mass{\Leb d(\Omega)}>0, \nonumber
\\
\frac 1{c_\eps^2} \int_0^1
\Phi_{\Mob,\Omega^\eps}(\mu^{\eps,\delta}_s,\nnu^{\eps,\delta}_s)\,\d
s \le \int_0^1 \Phi_{\Mob,\Omega}(\mu_s,\nnu_s)\,\d s=
\lim_{\eps,\delta\downarrow0}\int_0^1
\Phi_{\Mob,\Omega^\eps}(\mu^{\eps,\delta}_s,\nnu^{\eps,\delta}_s)\,\d
s,\nonumber
\end{gather}
where $c_\eps:=1+2\eps$.
\end{lemma}
\begin{proof}
Let us extend $(\mu_s,\nnu_s)$ outside the unit interval by
setting $\nnu_s\equiv0$ and $\mu_s\equiv \mu_0$ if $s<0$,
$\mu_s\equiv \mu_1$ if $s>1$; it is immediate to check that
$(\mu,\nnu)$ still satisfy the continuity equation. We then
consider a family of smooth and convex open sets $\Omega^\eps$ satisfying
$\Omega+B_{2\eps}(0)\subset \Omega^\eps\subset
\Omega+B_{3\eps}(0)$ and define $\tilde\mu^\eps_s:=\mu\ast
k_\eps,\tilde\nnu^\eps_s:=\nnu\ast k_\eps$ which have smooth
densities $\tilde\rho^\eps_s,\tilde\ww^\eps_s$
and are concentrated in $\overline\Omega+B_\eps(0)$. We
perform a further time convolution with respect to a
$1$-dimensional family of nonnegative smooth mollifiers
$h_\eps(z):=\eps^{-1}h(z/\eps)$ with support in $[-\eps,\eps]$ and
integral $1$
\begin{displaymath}
\bar\mu^\eps_s:=\int_\R \tilde\mu^\eps_z h_\eps(s-z)\,\d z,\quad
\bar\nnu^\eps_s:=\int_\R \tilde\nnu^\eps_z h_\eps(s-z)\,\d z,
\end{displaymath}
with corresponding densities $\rho^\eps_s,\ww^\eps_s$. Notice that
$\bar\mu^\eps_{-\eps}=\mu^{\eps,0}_0,\bar\mu^\eps_{1+\eps}=\mu^{\eps,0}_1$
and, by the convexity of $\phi_\Mob$ and Jensen's inequality, we have
\begin{equation*}
\label{eq:91} \phi_{\Mob}(\rho^\eps_s,\ww^\eps_s)\le \int_\R
\phi_{\Mob}(\tilde\rho^\eps_z, \tilde\ww^\eps_z)h_\eps(s-z)\,\d z,\quad
\Phi_{\Mob,\Omega^\eps}(\bar\mu^\eps_s,\bar\nnu^\eps_s)\le \int_\R
\Phi_{\Mob,\Omega^\eps}(\tilde\mu^\eps_s,\tilde\nnu^\eps_s)h_\eps(s-z)\,\d z
\end{equation*}
so that, being $\bar\nnu^\eps_s=0$ if $s<-\eps$ or $s>1+\eps,$
\begin{equation*}
\label{eq:92}
\int_{-\eps}^{1+\eps}\Phi_{\Mob,\Omega^\eps}(\bar\mu^\eps_s,\bar\nnu^\eps_s)\,\d
s = \int_\R
\Phi_{\Mob,\Omega^\eps}(\bar\mu^\eps_s,\bar\nnu^\eps_s)\,\d s \le
\int_\R \Phi_{\Mob,\Omega^\eps}(\tilde\mu^\eps_s,\tilde\nnu^\eps_s)\,\d s \le
\int_0^1 \Phi_{\Mob,\Omega}(\mu_s,\nnu_s)\,\d s.
\end{equation*}
We eventually set
\begin{equation*}
\label{eq:93} \mu^\eps_s:=\bar\mu^\eps_{c_\eps s-\eps},\quad
\nnu^\eps_s:=c_\eps\bar\nnu^\eps_{c_\eps s-\eps},\quad
c_\eps:=1+2\eps
\end{equation*}
and
\begin{equation*}
\label{eq:94}
\mu^{\eps,\delta}_s:=(1-\delta)\mu^\eps_s+\delta\lambda^\eps,\quad
\nnu^{\eps,\delta}_s:=\nnu^\eps_s .
\end{equation*}
It is then easy to check that all the requirements are satisfied.
\end{proof}

\subsection{Couple of measures at finite $\Wmod_{\Mob,\Omega}$ distance}
\label{subsec:finite}

We discuss now some cases when it is possible to prove that the
distance between two measures is finite. We already know
\cite[Cor. 5.25]{DNS} (in the case A, but the same argument can be
easily adapted to cover the case $M<+\infty$) that when $\Omega$ is convex and bounded
\begin{equation}
\label{eq:116} \text{if }\mu_i=\rho_i\Leb d\text{ with
$\|\rho_i\|_{L^\infty(\Rd)}<M$}
\text{ then }\Wmod_{\Mob,\Omega}(\mu_0,\mu_1)<\infty.
\end{equation}
We focus on the case A, $M=+\infty$, and exploit some ideas of
\cite{Sa08}. In order to refine the condition \eqref{eq:116}, we
first introduce the functions
\begin{equation*}
\label{eq:71} k_{\Mob,d}(r):=
\Big(r^{1+1/d}\Mob(r)\Big)^{-1/2},\quad K_{\Mob,d}(r):=\frac
1d\int_r^{+\infty} k_{\Mob,d}(z)\,\d z,\quad r>0.
\end{equation*}
Observe that $K_{\Mob,d}$ is either everywhere finite or
identically $+\infty$. In particular, in the case
$\Mob(r)=r^\alpha$, $K_{\Mob,d}$ is finite if and only if
$\alpha>1-1/d$.
\begin{theorem}
\label{thm:finite2} Let $\Omega$ be a bounded, open convex set of $\R^d$.
Suppose that $M=+\infty$, $\mass>0$, and that
$K_{\Mob,d}$ is finite. Then any two measures $\mu_0,\mu_1\in
\MM^+(\overline\Omega,\mass)$
have finite distance $\Wmod_{\Mob,\Omega}(\mu_0,\mu_1)<+\infty$
and the topology induced by $\Wmod_{\Mob,\Omega}$ on the space
$\MM^+(\overline\Omega,\mass)$ coincides with the usual weak$^*$
topology. In particular, the metric space
$(\MM^+(\overline\Omega,\mass),\Wmod_{\Mob,\Omega})$ is compact
and separable.
\end{theorem}
\begin{proof}
We fix an open set $B$ with compact closure in $\Omega$ and a
reference measure $\lambda=\bar\rho\Leb d\restr B$ with
$\lambda(\Omega)=\mass$ and $0<\bar\rho(x)\le b$ for $\Leb
d$-a.e.\ $x$ in $B$. Since $\Wmod_{\Mob,\Omega}$ satisfies the
triangular inequality, the first part of the theorem follows if we
show that $\Wmod_{\Mob,\Omega}(\lambda,\mu)<+\infty$ for every
$\mu\in \MM^+(\overline\Omega,\mass)$.

Let $\rr:B\to \overline\Omega$ be the Brenier map pushing
$\lambda$ onto $\mu$: we know that $\rr$ is cyclically monotone.
We set $\rr_s:=(1-s)\ii+s \rr$ with image $B_s\subset
(1-s)\overline B+s\overline\Omega\subset \overline\Omega$ and
inverse $\ss_s=\rr_s^{-1}:B_s\to B$, and $\vv_s:=(\rr-\ii)\circ
\rr_s^{-1}=\ii-\ss_s$. It is well known that $\ss_s$ is a
Lipschitz map with Lipschitz constant bounded by $(1-s)^{-1}$ and
that the curve $\mu_s:=(\rr_s)_\#\lambda$ belongs to
$\CE_\Omega(\lambda\to\mu)$ with
\begin{equation*}
\label{eq:67} \mu_s=\bar\rho_s\Leb d\restr{B_s},\quad
\bar\rho_s=\nchi_{B_s}\,\bar\rho(\ss_s)J_s,\quad J_s:=\det \mathrm
D\ss_s,\quad \nnu_s=\bar\rho_s\vv_s\Leb d.
\end{equation*}
Since the map $r\mapsto r/\Mob(r)$ is nondecreasing and $J_s\le
(1-s)^{-d}$, it follows that
\begin{align}
\label{eq:63} \Phi_{\Mob,\Omega}(\mu_s,\nnu_s)&= \int_{B_s}
\frac{\bar\rho_s^2}{\Mob(\bar\rho_s)}|\vv_s|^2\,\d x =\int_B
\frac{\bar\rho(y) J_s(\rr_s(y))}
{\Mob\big(\bar\rho(y)J_s(\rr_s(y)\big)}
|\rr(y)-y|^2\bar\rho(y)\,\d y\\
\label{eq:69} &\le\frac{b(1-s)^{-d}}{\Mob(b(1-s)^{-d})}\int_B
|\rr(y)-y|^2\bar\rho(y)\,\d y=
\frac{b(1-s)^{-d}}{\Mob(b(1-s)^{-d})}W_2^2(\lambda,\mu).
\end{align}
Taking the square and applying \eqref{eq:70}, since
\begin{displaymath}
\int_0^1 \left(\frac
{b(1-s)^{-d}}{\Mob(b(1-s)^{-d})}\right)^{1/2}\,\d s=
\frac{b^{1/d}}{d}\int_b^{+\infty}\left(\frac
z{\Mob(z)}\right)^{1/2} z^{-1-1/d}\,\d z = b^{1/d}\,K_{\Mob,d}(b)
\end{displaymath}
we get the estimate
\begin{equation}
\label{eq:68} \Wmod_{\Mob,\Omega}(\lambda,\mu)\le
b^{1/d}\,K_{\Mob,d}(b)\,W_2(\lambda,\mu).
\end{equation}
A completely analogous calculation with $\mu:=\mu_0$
(resp.~$\mu:=\mu_1$) and $\mu_s=\mu_{0,s}$
(resp.~$\mu_s=\mu_{1,s}$) shows that
\begin{equation*}
\label{eq:72} \Wmod_{\Mob,\Omega}(\mu_{i,1-\eps},\mu_i)\le b^{1/d}
K_{\Mob,d}\big(b\eps^{-d}\big)\,W_2(\lambda,\mu_i)\quad\forall\,\eps>0,\quad
i=0,1.
\end{equation*}
On the other hand, taking into account that the density of
$\mu_{i,1-\eps}$ is bounded by $b\eps^{-d}$, we can apply
\eqref{eq:68} with $\mu_{0,1-\eps}$ instead of $\lambda$,
obtaining
\begin{equation*}
\label{eq:73}
\Wmod_{\Mob,\Omega}(\mu_{0,1-\eps},\mu_{1,1-\eps})\le
b^{1/d}\eps^{-1}\,
K_{\Mob,d}(b\eps^{-d})\,W_2(\mu_{0,1-\eps},\mu_{1,1-\eps}).
\end{equation*}
Therefore, the triangular inequality yields
\begin{equation*}
\label{eq:74} \Wmod_{\Mob,\Omega}(\mu_0,\mu_1)\le
b^{1/d}\,K_{\Mob,d}(b\eps^{-d})\,\Big(W_2(\mu_0,\lambda)+W_2(\mu_1,\lambda)+
\eps^{-1}W_2(\mu_{0,1-\eps},\mu_{1,1-\eps})\Big).
\end{equation*}
Applying this estimate to a sequence $\mu_n$ weakly$^*$ converging
to $\mu$ (and therefore converging also with respect to $W_2$),
since the corresponding geodesic interpolants with $\lambda$
$\mu_{n,1-\eps}$ converge to $\mu_{1-\eps}$ as $n\to\infty$ with
respect to $W_2$, we easily obtain
\begin{equation*}
\label{eq:75} \limsup_{n\to\infty}
\Wmod_{\Mob,\Omega}(\mu_n,\mu)\le
2b^{1/d}\,K_{\Mob,d}(b\eps^{-d})W_2(\mu,\lambda).
\end{equation*}
Since $\lim_{\eps\downarrow0}K_{\Mob,b}(b\eps^{-d})=0$, taking
$\eps$ arbitrarily small, we conclude.
\end{proof}

In the next result we do not assume any particular condition on
$\Mob$, but we ask that $\mu_i\ll\Leb d$ with densities satisfying
some extra integrability assumptions.

\begin{theorem}
\label{thm:finite1} Let $\Omega$ be a bounded, open convex set of
$\R^d$ and assume that $M=+\infty$, $\mass>0$. If the measures
$\mu_i=\rho_i\Leb d\restr\Omega\in \MM^+(\Omega,\mass)$, $i=0,1$,
satisfy
\begin{equation}
\label{eq:76} \int_\Omega \frac{\rho_i(x)^2}{\Mob(\rho_i(x))
}
\,\dx<+\infty\quad i=0,1,
\end{equation}
then $\Wmod_{\Mob,\Omega}(\mu_0,\mu_1)<+\infty.$
\end{theorem}
\begin{proof}
We argue as in the previous proof, keeping the same notation and
observing that for $0\le s\le 1/2$ \eqref{eq:69} yields
\begin{equation}
\label{eq:78} \Phi_{\Mob,\Omega}(\mu_s,\nnu_s)\le \frac{b\, 2^d}{\Mob(b\,
2^d)}W^2_2(\lambda,\mu).
\end{equation}
When $1/2\le s\le 1$ we invert the role of $\lambda$ and
$\mu=\rho\Leb d$ in \eqref{eq:63} obtaining
\begin{equation}
\label{eq:79} \Phi_{\Mob,\Omega}(\mu_s,\nnu_s)= \int_\Omega \frac{\rho(y)
\tilde J_s(\tilde \ss_s(y))} {\Mob\big(\rho(y)\tilde J_s(\tilde
\ss_s(y)\big)} |\tilde \ss(y)-y|^2\rho(y)\,\d y
\end{equation}
where $\tilde \ss_s=(1-s)\ss+s\ii$ is the optimal map pushing
$\mu$ onto $\mu_s$ and $\tilde J_s=\det \mathrm D
\tilde\ss_s^{-1}$ satisfies $\tilde J_s\le s^{-d}$. \eqref{eq:79}
then yields for $1/2\le s\le 1$
\begin{equation}
\label{eq:80} \Phi_{\Mob,\Omega}(\mu_s,\nnu_s)\le 2^{d+1}\int_\Omega
\frac{\rho(y)^2} {\Mob\big(\rho(y))} \Big(|\tilde
\ss(y)|^2+|y|^2\Big)\,\d y.
\end{equation}
Since the range of $\tilde\ss(y)$ is $\mu$-essentially bounded,
the integral in \eqref{eq:80} is finite thanks to \eqref{eq:76}.
Integrating \eqref{eq:78} in $(0,1/2)$ and \eqref{eq:80} in
$(1/2,1)$ we conclude that $\Wmod_{\Mob,\Omega}(\lambda,\mu)$ is
finite.
\end{proof}


\section{Geodesic convexity of integral functionals
and their gradient flows}

\subsection{Nonlinear diffusion equations: weak and limit solutions}
We consider a
\begin{subequations}
\begin{equation}
\label{eq:81} \text{convex \emph{density function} $U\in
W^{2,1}_\loc(0,M)$ with $\Mob U''\in L^1_\loc([0,M))$ }
\end{equation}
and a \emph{pressure function} $P:[0,M)\to \R$ defined by
\begin{equation*}
\label{eq:64} P(r):=\int_0^r \Mob (z) U''(z)\,\d z.
\end{equation*}
\end{subequations}
Let us observe that $P\in W^{1,1}_\loc([0,M))$ is nondecreasing,
continuous, and $P(0)=0$. When $U$ has a superlinear growth at
$+\infty$ the corresponding \emph{internal energy functional} $\UU
:D(\UU)\subset \MM^+_\comp(\Rd)\to (-\infty,+\infty]$ is defined
as
\begin{equation}
\label{eq:65} \UU(\mu):= \displaystyle\int_\Rd U(\rho(x))\,\dx
,\quad D(\UU):=\Big\{\mu=\rho\Leb d\in \MM^+_\comp(\Rd):U(\rho)\in
L^1(\Rd)\Big\}
\end{equation}
Since $U$ is bounded from below by a linear function and $\mu$ has
compact support, the integral in \eqref{eq:65} is always well
defined. $\UU$ is lower semicontinuous with respect to weak
convergence in $\MM^+_\comp(\Rd)$ if and only if
\begin{equation*}
\label{eq:117} U'(+\infty):=\lim_{r\uparrow+\infty}\frac {U(r)}r=
\lim_{r\uparrow+\infty}U'(r)=+\infty.
\end{equation*}
When $U'(+\infty)<+\infty$ we define the functional $\UU$ as
\begin{equation*}
\label{eq:118} \UU(\mu):= \displaystyle\int_\Rd
U(\rho)\,\dx+U'(+\infty)\mu^\perp(\Rd),\quad \mu=\rho\Leb
d+\mu^\perp,
\end{equation*}
where $\mu^\perp$ is the singular part of $\mu$ in the usual
Lebesgue decomposition.

Let $\Omega\subset\R^d$ be a bounded, open, and connected set with
Lipschitz boundary $\partial\Omega$ and exterior unit normal
$\nn$. We will often suppose that $\Omega$ is convex in the
sequel. We consider the homogeneous Neumann boundary value problem
for the nonlinear diffusion equation
\begin{equation}\label{PME}
\de_t\rho-\Delta P(\rho) = 0 \quad \mbox{in } (0,+\infty)\times
\Omega,\qquad
\partial_\snn P(\rho)=0\quad
\text{on }(0,+\infty)\times\de\Omega,
\end{equation}
with \emph{nonnegative} initial condition $\rho(0,\cdot)=\rho_0$.
We also introduce the dissipation rate of $\UU$ along the flow by
\begin{equation}
\label{eq:122} \DD(\rho)=\int_\Omega \frac{|\nabla
P(\rho)|^2}{\Mob(\rho)}\,\dx= \int_\Omega
\phi_{\Mob}(\rho,\nabla P(\rho))\,\d x , \quad \forall\, 0\le \rho\in
L^1(\Omega),\ P(\rho)\in W^{1,1}(\Omega).
\end{equation}
We collect in the following result some well established facts
\cite{Vaz} on weak and classical solutions to \eqref{PME}.

\begin{theorem}[Very weak and classical solutions]
\label{thm:weak_classical} Let us suppose that $\Omega$ is bounded
and $\rho_0\in L^\infty(\Omega)$. There exists a unique solution
$\rho\in L^\infty((0,+\infty)\times \Omega) \cap
C^0([0,+\infty);L^1(\Omega))$ with $P(\rho)\in
L^\infty((0,+\infty)\times \Omega) \cap
L^2((0,+\infty);W^{1,2}(\Omega))$ to \eqref{PME} satisfying the
following weak formulation
\begin{equation}
\label{eq:59} \int_0^{+\infty}\int_\Omega \Big(\rho\partial_t
\zeta-\nabla P(\rho)\cdot \nabla\zeta\Big)\,\dx\,\d t=0
\quad\forall\, \zeta\in C^\infty_\comp((0,+\infty)\times\Rd),
\end{equation}
and the initial condition $\rho(0,\cdot)=\rho_0$. The energy $\UU$
is decreasing along the flow and satisfies the identity
\begin{equation}
\label{eq:120} \int_\Omega U(\rho(T,x))\,\dx+\int_0^T \int_\Omega
\frac{|\nabla P(\rho)|^2}{\Mob(\rho)} \,\dx\, \d t= \int_\Omega
U(\rho_0(x))\,\dx, \qquad \forall \,\, T>0.
\end{equation}
The map $\rho_0\mapsto S_t \rho_0:=\rho(t,\cdot)$ can be extended
to a $C^0$ contraction semigroup $\Semi=\Semi(P,\Omega)$ in the
positive cone of $L^1(\Omega)$, whose curves $\Semi_t\rho_0$ are
also called ``limit $L^1$-solutions'' of \eqref{PME}, and it
satisfies
\begin{equation*}
\label{eq:60} \mathrm{ess\,inf}_\Omega \,\rho_0\le S_t\rho_0\le
\mathrm{ess\,sup}_\Omega \,\rho_0.
\end{equation*}
If moreover $U,\Mob\in C^\infty(0,M)$, $U$ is uniformly convex,
$\Omega$ is smooth and $\inf_\Omega \rho_0>0$, then $\rho\in
C^\infty((0,+\infty)\times \overline\Omega)$ and is a classical
solution to \eqref{PME}.
\end{theorem}

Let us briefly discuss here two useful lemma, whose proof follows
from a standard variational argument.
\begin{lemma}
\label{le:weak-very} If $\rho_0$, $U(\rho_0)\in L^1(\Omega)$
then the limit $L^1$-solution $\rho=\Semi(\rho_0)$ satisfies
$P(\rho) \in L^1_\loc([0,+\infty);W^{1,1}(\Omega))$, the weak
formulation \eqref{eq:59}, and the energy inequality
\begin{equation}
\label{eq:142} \int_\Omega U(\rho(T,x))\,\dx+\int_0^T \int_\Omega
\frac{|\nabla P(\rho)|^2}{\Mob(\rho)} \,\dx\, \d t\le \int_\Omega
U(\rho_0(x))\,\dx.
\end{equation}
\end{lemma}
\begin{proof}
Let us first show that we can find a constant $C$ depending only
on $P$, $\omega:=\Leb d(\Omega)$, $\mass=\int_\Omega \rho\,\d x$,
and the constant $c_p$ in the Poincar\'e inequality for $\Omega$
such that
\begin{equation}
\label{eq:138} \|P(\rho)\|_{L^1(\Omega)}\le C \Big(1+\|\nabla
P(\rho)\|_{L^1(\Omega)}\Big)\quad \forall\, \rho\in L^1(\Omega),\
\int_\Omega \rho\,\d x=\mass,\ P(\rho)\in W^{1,1}(\Omega).
\end{equation}
In fact, setting $\mathsf p:=\int_\Omega P(\rho)\,\d x$ and
$\ell:=\Leb d(\{x\in \Omega: P(\rho)\ge \mathsf p/2\})$ Poincar\'e
and Chebyshev inequality yield
\begin{equation*}
\label{eq:141} \frac 12\mathsf p(\omega-\ell)\le \int_\Omega
|P(\rho)-\mathsf p|\,\d x\le c_p\int_\Omega |\nabla P(\rho)|\,\d
x,\quad \frac 12\mathsf p\le P(\mass/\ell),
\end{equation*}
so that if $\ell\ge \omega/2$ we get $\mathsf p\le
2P(2\mass/\omega)$, whereas if $\ell\le \omega/2$ we obtain
$\mathsf p\le 4\omega^{-1}c_p \int_\Omega |\nabla P(\rho)|\,\d x.$

If now $\rho_t=\Semi_t \rho_0$ is the $L^1(\Omega)$-limit of a
sequence $\rho_{n,t}=\Semi_t \rho_{n,0}$ of bounded solutions with
$U(\rho_{n,0})\to U(\rho_0)$ in $L^1(\Omega)$ as
$n\uparrow+\infty$, from the uniform bound \eqref{eq:120} we
obtain for every bounded Borel set $\mathcal T\subset
(0,+\infty)$, every $B\subset \Omega$, and every nonnegative
constants $a,b$ such that $\Mob(r)\le a+br$,
\begin{align*}
\int_{\mathcal T}\int_B |\nabla P(\rho_n)|\,\d x\,\d t&\le \|\Mob
(\rho_n)\|_{L^1(\mathcal T\times B)}^{1/2} \Big( \int_{\mathcal
T\times B} \frac{|\nabla P(\rho_n)|^2}{\Mob(\rho_n)}\,\d x\,\d
t\Big)^{1/2}
\\&\le C
\|a+b\rho_n\|_{L^1(\mathcal T\times B)}^{1/2}.
\end{align*}
Taking $\mathcal T=(0,T), B=\Omega$ and applying \eqref{eq:138},
we obtain a uniform bound of the sequence $P(\rho_n)$ in
$L^1(0,T;W^{1,1}(\Omega))$; since $\rho_n$ converges to $\rho$ in
$L^1((0,T)\times \Omega)$, we obtain that $\nabla P(\rho_n)$ is
uniformly integrable and therefore it converges weakly to $\nabla
P(\rho)$ in $L^1((0,T)\times \Omega)$. It follows that $P(\rho)\in
L^1(0,T;W^{1,1}(\Omega))$ and we can then pass to the limit in the
weak formulation \eqref{eq:59} written for $\rho_n$, obtaining the
same identity for $\rho$. The inequality \eqref{eq:142} eventually
follows by the same limit procedure, recalling that the
dissipation functional \eqref{eq:122} is lower semicontinuous with
respect to weak convergence in $L^1(\Omega)$.
\end{proof}

The following stability result is used in the sequel; its proof is
an easy adaption of \cite[Prop. 6.10]{Vaz}.

\begin{proposition}
\label{prop:stability} Let $\Omega^n\subset \Rd$ be a decreasing
sequence of open, bounded, convex sets converging to $\Omega$ and
let $\Semi^n=\Semi(P,\Omega^n),\ \Semi(P,\Omega)$ be the
associated semigroups provided by Theorem
\ref{thm:weak_classical}. If (after a trivial extension to $0$
outside $\Omega^n$) $\rho^n_0\in L^1(\Omega^n)$ is converging
strongly in $L^1(\Rd)$ to $\rho_0\in L^1(\Omega)$, then
$\Semi^n_t(\rho^n_0)\to \Semi_t(\rho_0)$ in the same $L^1$ sense,
as $n\uparrow+\infty$ for every $t>0$.
\end{proposition}

\subsection{The generalized McCann condition}
We assume that $P'\Mob '\in L^1_\loc([0,M))$ and we introduce a
primitive function $H$ of $h:=P'\Mob'=U''\Mob\Mob',$
\begin{equation}
\label{eq:62} H(r):=H_0+\int_{0}^r P'(z)\Mob'(z)\,\d z\quad
\text{for some }H_0\ge 0.
\end{equation}
When the dimension $d$ is greater than $1$, we assume that
$\inf_{r\in(0,M)}H(r)=0$; this means that the locally
integrability assumption on $h$ cannot be avoided, as well as, in
the case when $M<+\infty$, its integrability in $(0,M)$. In the
case $M=+\infty$ we can simply choose $H_0=0$ in \eqref{eq:62}. In
the case $M<+\infty$ we can choose
\begin{equation}
\label{eq:20} H_0:=\Big(\int_0^M P'\Mob'\,\d x\Big)^-<+\infty.
\end{equation}
\begin{remark}
\label{rem:previous_ass} \upshape Notice that in the most common
case when $\Mob'(0_+)=\lim_{r\downarrow0}r^{-1}\Mob(r)>0$, the
local integrability of $h$ in a right neighborhood of $0$ implies
the local integrability of $\Mob\,U''$, as we already required in
\eqref{eq:81}, and the lower boundedness of $P$. When the space
dimension is $1$ all these restriction can be removed: we comment
on this issue in the next remark \ref{rem:dim1}. If $M<+\infty$,
$P$ is locally Lipschitz near $0$ and $M$, and
$\Mob(0)=\Mob(M)=0$, then we get $H_0=0$ so that $P$ and $\Mob$
should satisfy the compatibility condition $\int_0^M P'\Mob'\,\d
r=0$.
\end{remark}

\begin{definition}
Let $U,P,H$ and $\Mob$ be defined in the interval $(0,M)$
according to (\ref{eq:81},b) and \eqref{eq:62}. We say that the
energy density $U$ and the corresponding pressure function $P$
satisfy the \emph{$d$-dimensional generalized McCann condition for
the mobility $\Mob$}, denoted by $\GMC\Mob d$, if for a suitable choice of
$H_0$
\begin{subequations}
\label{eq:gmcab}
\begin{equation}\label{hp:gmc}
U''(r)\Mob^2(r)=P'(r)\Mob(r)\ge (1-{1}/{d}) H(r)\geq 0, \quad
\forall r\in (0,M),
\end{equation}
or, equivalently,
\begin{equation*}
\label{eq:86bis} r\mapsto \frac{H(r)}{\Mob^{1-1/d}(r)}=\frac
1{\Mob^{1-1/d}(r)} \int_0^r P'(s)\Mob'(s)\,\d s \quad \text{is
nondecreasing in }(0,+\infty).
\end{equation*}
\end{subequations}
\end{definition}
We collect in the following remarks some simple properties related
to this definition.
\begin{remark}[Elementary properties]
\label{rem:various} \upshape \
\begin{enumerate}
\item (Linear mobility) \eqref{hp:gmc} is consistent with the
usual McCann condition \eqref{mc} in the linear case of $\Mob(r)=r$.

\item (Dimension $d=1$) As in the case of McCann condition, in
space dimension $d=1$ \eqref{hp:gmc} is equivalent to the
convexity of $U$ or to the monotonicity of $P$.

\item (Local boundedness of $U$ when $d>1$). In dimension $d>1$
the energy density function $U$ is bounded in a right neighborhood
of $0$ (and in a left neighborhood of $M$, in the case
$M<+\infty$). Since $U''=P'/\Mob$ the property is immediate if
$\Mob(0)>0$. If $\Mob(0)=0$ then $\Mob'(0)>0$ and therefore $P$ is
bounded around $0$ and the formula
\begin{displaymath}
U(r)=U(r_0)+U'(r_0)(r-r_0)+\int_{0}^{r_0} \frac{(z-r)^+}{\Mob(z)}
P'(z)\,\d z,\quad r\in (0,r_0],
\end{displaymath}
shows that $\lim_{r\downarrow0}U(r)<+\infty$.

\item (Constant mobility) When $\Mob(r)\equiv c>0$ \eqref{hp:gmc}
is still equivalent to the convexity of $U$.

\item (The power-like case) In the case of $P(r)=r^\gamma$
($\gamma=\alpha+\beta-1$ if $U(r)=r^\beta$) and
$\Mob(r)=r^\alpha$, \eqref{hp:gmc} is satisfied if and only if
\begin{equation}\label{cpme}
\gamma\geq 1-\frac{\alpha}{d}.
\end{equation}

\item (The case $P(r)=r$) It is immediate to check that the couple
$(r,\Mob)$ always satisfies \eqref{hp:gmc}: it corresponds to the
entropy function $U_\Mob$ whose second derivative is $\Mob^{-1}$.
After fixing some $r_0\in (0,M)$ (the choice $r_0=0$ is admissible
if $\Mob^{-1}$ is integrable in a right neighborhood of $0$), we
obtain
\begin{equation*}
\label{eq:22} U_\Mob(r):=\int_{r_0}^r \frac{r-z}{\Mob(z)}\,\d
z,\quad P_\Mob(r)=r-r_0.
\end{equation*}

\item (The case of the logarithmic entropy) $U(r)=r\log r$
satisfies $\GMC{r^\alpha}d$ if and only if $ \gamma=\alpha\ge
d/(d+1)$.

\item (Linearity) If $P_1$ and $P_2$ satisfy
$GMC(\Mob,d)$ then also $\alpha_1P_1+\alpha_2P_2$ satisfies $GMC(\Mob,d)$, for
every $\alpha_1,\alpha_2\ge0$.
Analogously, if $P$ satisfies $GMC(\Mob_1,d)$ and
$GMC(\Mob_2,d)$ then
$P$ satisfies $GMC(\alpha_1\Mob_1+\alpha_2\Mob_2,d)$.
In particular, if $P$ satisfies $GMC(\Mob,d)$ then $P(r)+\alpha r$
satisfies $GMC(\Mob+\beta,d)$ for every $\alpha,\beta\ge0$.

\item (Shift) If $M=+\infty$ and $P$ satisfies $GMC(\Mob,d)$
then $P$ satisfies $GMC(\Mob(\cdot+\alpha),d)$ and
$P(\cdot-\alpha)$ satisfies $GMC(\Mob,d)$, for every $\alpha\ge0$.
\end{enumerate}
\end{remark}

The next two properties are more technical and require a detailed
proof.

\begin{lemma}[Smoothing]
\label{le:smoothing_2} Let us assume that $P$ satisfies $\GMC\Mob
d$ and let us fix two constants $0<M'<M''<M$. Then there exists a
family $P_\eta ,\Mob_\eta$, $\eta>0$, with smooth restriction to
$[M',M'']$ such that $P_\eta\ge P$ is strictly increasing,
$\Mob_\eta\ge \Mob$ is concave, $P_\eta$ satisfies
$\GMC{\Mob_\eta} d$ (in $[0,M'']$), and $P_\eta,\Mob_\eta$
converge uniformly to $P,\Mob$ in $[M',M'']$ as $\eta\downarrow
0$. Moreover, if $P'$ is locally integrable in a right
neighborhood of $0$, then we can choose $M'=0$.
\end{lemma}
\begin{proof}
When $M'>0$ it is not restrictive (up to choosing a smaller $M'$)
to assume that $M'$ is a Lebesgue point of the derivative of $P$.
Let $H$ be as in \eqref{eq:62} and let us set
$\tilde\Mob_\eta(r):=\Mob(r)+\eta$, $\tilde P_\eta(r)=P(r)+\eta
r$,
\begin{displaymath}
\tilde H_\eta(r)=H_0+\eta\Mob(0)+\eta^2+\int_0^r \tilde
P_\eta'(r)\tilde\Mob_\eta'(r)\,\d r= H(r)+\eta
\Mob(r)+\eta^2\ge\eta^2>0.
\end{displaymath}
By the previous Remark (points 6 and 8) $\tilde P_\eta$ satisfies
$\GMC{\tilde\Mob_\eta} d$ and moreover
\begin{equation}
\label{eq:106} \tilde P_\eta'\tilde\Mob_\eta-(1-1/d)\tilde H_\eta=
P'\Mob-(1-1/d) H+\frac\eta d(\eta + \Mob)\ge \frac\eta d
(\Mob+\eta) \ge \eta^2/d.
\end{equation}
By choosing a family of mollifiers $h_\delta$, $\delta>0$, with
support in $[0,\delta]$, we introduce the functions
\begin{align*}
\tilde P_{\eta,\delta}(r):=&
\begin{cases}
\tilde P_\eta(M')+\int_{M'}^r \tilde P_\eta'\ast h_\delta \,\d
s&\text{if }r\ge M'.\\
\tilde P_\eta(r)&\text{if }r<M',
\end{cases}\\
\tilde \Mob_{\eta,\delta}(r):=&
\begin{cases}
\tilde \Mob_\eta(M')+\int_{M'}^r \tilde \Mob_\eta'\ast h_\delta
\,\d
s&\text{if }r\ge M',\\
\tilde \Mob_\eta(r)&\text{if }r<M',
\end{cases}
\end{align*}
which are smooth in $[M',M'']$ and satisfy the requested
monotonicity/concavity conditions. Since $\tilde P_{\eta,\delta}'$
converges to $\tilde P_\eta'$ in $L^1_{\loc}(0,M']$ and $\tilde
\Mob_{\eta,\delta} $ is uniformly bounded and converges pointwise
a.e. to $\tilde \Mob_{\eta}'$ as $\delta\to 0$, we conclude that
the corresponding continuous functions $\tilde H_{\eta,\delta}$
converge uniformly to $\tilde H_\eta$ as $\delta\downarrow0$. By
\eqref{eq:106}, we can find a sufficiently small
$\delta=\delta_\eta$ depending on $\eta$ such that
\begin{displaymath}
\tilde P_{\eta,\delta_\eta}'\tilde\Mob_{\eta,\delta_\eta}'\ge
(1-1/d)\tilde H_{\eta,\delta_\eta}\ge0.
\end{displaymath}
A standard diagonal argument concludes the proof.
\end{proof}
\begin{lemma}[Minimal asymptotic behaviour]
When $d>1$ and $M=+\infty$, the function $P_{\mathrm {min}}(r):=
\int_0^r \Mob(z)^{-1/d}\,\d z$ satisfies $GMC(\Mob,d)$ and provides an
(asymptotic) lower bound for every any other $P$, since for every
$r_0>0$ there exists a constant $c_0>0$ such that
\begin{equation}
\label{eq:84} P'(r)\ge c_0 P_{\mathrm{min}}'(r)=
c_0\Mob(r)^{-1/d},\quad U''(r)\ge c_0 \Mob(r)^{-1-1/d}\quad
\text{for a.e.\ }r\ge r_0.
\end{equation}
\end{lemma}
\begin{proof}
In fact $f(r):=P'\Mob$ satisfies
\begin{equation*}
\label{eq:82} f(r)\ge (1-1/d)\Big(H(r_0)+\int_{r_0}^r
f(r)\Mob'(r)/\Mob(r)\,\d r\Big)
\end{equation*}
Gronwall Lemma then yields \eqref{eq:84} with $c_0:=
(1-1/d)H(r_0)\Mob(r_0)^{1/d-1}$.
\end{proof}

Notice that in the case $\Mob(r)=r^\alpha$ we obtain the functions
$P_{\mathrm{min}}(r) =c\,r^{\gamma_0}$ with exponent
$\gamma_0=1-\alpha/d$, which is consistent with \eqref{cpme}. The
corresponding energy density functions are then
$U_{\mathrm{min}}(r)= c r^{2-\alpha(1+1/d)}$: in particular, when
$\alpha < d/(d+1)$, all the energy functions have a superlinear
growth as $r\uparrow\infty$.

\begin{remark}[A sufficient condition]
\label{rem:easier} \upshape It is possible to give a simpler
sufficient condition than \eqref{hp:gmc}, at least when $\Mob U''$
is integrable in a right neighborhood of $0$ and $M=+\infty$: if
\begin{equation}
\label{eq:66} \text{the map } r\mapsto \Mob^{1/d}(r) P'(r)=
\Mob^{1+1/d}(r) U''(r) \text{ is positive and nondecreasing in
}(0,+\infty)
\end{equation}
then \eqref{hp:gmc} is satisfied. In fact, assuming $U$ smooth for
simplicity, \eqref{eq:66} is equivalent to
\begin{displaymath}
0\le \Mob^{1/d}P''-(1-1/d)\Mob^{1/d-1}\Mob' P' .
\end{displaymath}
Multiplying the inequality by $\Mob^{1-1/d}$ and integrating in
time we get \eqref{hp:gmc}. Condition \eqref{eq:66} gives the same
sharp bound \eqref{cpme} in the power case.
\end{remark}

\subsection{The metric approach to gradient flows}
We recall here some basic facts about the metric notion of
gradient flows, referring to \cite{AGS} for further details. Let
$(\DenseDom,\W)$ be a metric space, not assumed to be complete,
and let $\VV:D(\VV)\to(-\infty,+\infty]$ be a lower semicontinuous
functional. A family of continuous maps $\sfS_t:\DenseDom
\to\DenseDom$, $t\ge0$, is a $C^0$-\emph{(metric) contraction
gradient flow of $\VV$} with respect to $\W$ if
\begin{subequations}
\label{eq:18}
\begin{gather} 
\sfS_{t+h}(u)=\sfS_{h}\big(\sfS_t(u)\big),\quad
\lim_{t\downarrow0}\sfS_t(u)=\sfS_0(u)=u\quad
\forall\, u\in \DenseDom,\ t,h\ge0,\nonumber \\
\label{eq:55}
\frac12\W
^2(\sfS_{t}(u),v)-
\frac12\W
^2(u,v) \le t\Big(\VV(v)-\VV(\sfS_{t}(u))\Big) \quad
\forall\,t>0,\ u\in \DenseDom,\ v\in D(\VV).
\end{gather}
\end{subequations}
Thanks to \cite[Prop.~3.1]{DS}, conditions (\ref{eq:18}a,b) imply
\begin{gather}
\text{$\sfS_t(\DenseDom)\subset D(\VV)$ $\forall\,t>0$
and the map $t\mapsto \VV(\sfS_t(u))$ is not increasing in $(0,+\infty)$,}\nonumber \\
\label{evi1} \frac{1}{2}\frac{\d^+}{\d t}
\W^2(\sfS_t( u),v)+ \VV(\sfS_t(u))
\leq \VV(v), \quad\forall\,u\in\DenseDom,\,v\in D(\VV),\, t\geq0,\\
\VV(\sfS_t(u))\le \VV(v)+\frac 1{2t}
\W^2(u,v) \quad \forall\, u\in \DenseDom,\ v\in D(\VV),\ t>0,\nonumber \\
\label{eq:47}
\W^2(\sfS_{t_1}(u),\sfS_{t_0}(u))\le
2(t_1-t_0)\Big(\VV(\sfS_{t_0}u)- \VV_{\rm inf}\Big) \quad
\forall\, u\in D(\VV),\ 0\le t_0\le t_1,\\
\W(\sfS_t(u),\sfS_t(v))\le \W(u,v)\quad \forall\, u,v\in
\DenseDom,\ t\ge0.\nonumber
\end{gather}
In \eqref{evi1} we used the usual notation
\begin{equation*}
\frac{\d^+}{\dt }\zeta(t) = \limsup_{h\to
0^+}\frac{\zeta(t+h)-\zeta(t)}{h}.
\end{equation*}
for every real function $\zeta:[0,+\infty)\to\R$. \\
The following
approximated convexity estimate \cite[Theorem 3.2]{DS} plays an
important role in the sequel.
\begin{theorem}[Approximated convexity]
\label{thm:crucial} Let us suppose that $\sfS$ is metric
contraction gradient flow of $\VV$ with respect to $\W$ according
to \emph{(\ref{eq:18}a,b)} and let $s\mapsto u_s\in D$, $s\in
[0,1]$, be a Lipschitz (``almost'' geodesic) curve such that
$u_0,u_1\in D(\VV)$ and
\begin{equation}
\label{eq:123} \W(u_r,u_s)\le L|r-s|\quad\forall\, r,s\in [0,1],
\quad L^2\le \W^2(u_0,u_1)+\delta^2.
\end{equation}
Then for every $s\in [0,1]$ and $t>0$, we have
\begin{equation*}
\label{eq:89} \VV(\sfS_t(u_s))\le (1-s)\VV(u_0)+s\VV(u_1)
+\frac{s(1-s)}{2t}\delta^2.
\end{equation*}
In particular, if $u_s$ is a minimal geodesic, i.e. \eqref{eq:123}
holds with $\delta=0$, then
\begin{equation*}
\label{eq:126} \VV(u_s)\le (1-s)\VV(u_0)+s\VV(u_1)\quad \forall\,
s\in [0,1].
\end{equation*}
\end{theorem}

\subsection{Main results}

We state our main result about the generation of a contractive
gradient flow of $\UU$ with respect to $\Wmod_{\Mob,\Omega}$.

\begin{theorem}[Contractive gradient flow]\label{th:main}
Let us assume that $\Omega$ is a bounded, convex open set, and
the functions $U,P,H$ satisfy the generalized McCann condition
$\GMC\Mob d$. For every reference measure $\sigma\in
\MM^+(\Omega)$ with finite energy $\UU(\sigma)<+\infty$ the
functional $\UU$ generates a unique metric contraction gradient
flow $\GF =\GF(\UU,\Mob,\Omega)$ in the space
\begin{equation*}
\label{eq:134} D:=\Big\{\mu\in \MM^+(\Omega): \mu\ll\Leb
d\restr\Omega,\ \Wmod_{\Mob,\Omega}(\mu,\sigma)<+\infty,\ \UU(
\mu)<+\infty\Big\}
\end{equation*}
endowed with the distance $\Wmod_{\Mob,\Omega}$.
Moreover $\GF$ is characterized by the formula $\GF_t\mu_0=\rho_t\Leb
d\restr\Omega$, where $\rho_t=\Semi_t \rho_0$ is a limit
$L^1$-solution of \eqref{PME}.
\end{theorem}

When $\Mob$ satisfies the finiteness condition of Theorem
\ref{thm:finite2} (in particular $\Mob(r)=r^\alpha$ with
$\alpha>1-1/d$) we obtain a much more refined result, which in
particular shows the continuous dependence of $\GF$ on the
weak$^*$ topology.
\begin{corollary}
\label{cor:finite3} Under the same assumptions on $\Omega$, $U,P$
of the previous theorem, if moreover $M=+\infty$ and $\Mob$
satisfies the finiteness condition of theorem \ref{thm:finite2},
then the semigroup $\GF$ can be uniquely extended to a contraction
semigroup on every convex set $\MM^+(\overline\Omega,\mass)$,
which is continuous with respect to the weak$^*$ convergence of
the initial data. If $U$ has a superlinear growth, then
$\GF_t(\mu_0)=\rho_t\Leb d\ll \Leb d\restr\Omega$ for every $t>0$
and $\rho_t$ is a weak solution of \eqref{PME} according to
\eqref{eq:59}.
\end{corollary}

We conclude this section with our main convexity result.

\begin{theorem}[Convexity]\label{th:conv}
Let us assume that $\Omega$ is a bounded convex open set, and the
functions $U,P,H$ satisfy the generalized McCann condition
$\GMC\Mob d$. For every $\mu_0,\mu_1\in \MM^+(\Omega)\cap D(\UU)$
with finite distance $\Wmod_{\Mob,\Omega}(\mu_0,\mu_1)<+\infty$
there exists a constant speed minimizing geodesic for $\Wmod_{\Mob
,\Omega}$, $\mu:[0,1]\to \MM^+(\Omega)$ connecting $\mu_0$ to
$\mu_1$ such that
\begin{equation}
\label{eq:133} \UU(\mu_s)\leq s\UU(\mu_1)+(1-s)\UU(\mu_0), \qquad
\forall s\in[0,1].
\end{equation}
\end{theorem}
\begin{remark}[Weak and strong convexity]
\label{rem:weak-strong} \upshape When \eqref{eq:133} holds for all
the (constant speed, minimizing) geodesics, the functional $\UU$
is called \emph{strongly} geodesically convex. When $\Mob$ is
strictly concave and has a sublinear growth (or $M<+\infty$) then
every two measures with finite $\Wmod_{\Mob,\Omega}$-distance can
be connected by a \emph{unique} geodesic \cite[Theorem
5.11]{DNS}, so that there is no difference between strong or weak
convexity and \eqref{eq:133} yields that the map $s\mapsto
\UU(\mu_s)$ is convex in $[0,1]$.
\end{remark}

\begin{remark}[Absolutely continuous measures]
\label{rem:conv_refine} \upshape Even when geodesics are not
unique, the proof of Theorem \ref{th:conv} shows in fact that
\eqref{eq:133} is satisfied by \emph{any} geodesic $\mu_s$ with
$\mu_s\ll\Leb d$ for every $s\in \Rd$, which surely exist if $U$
has a superlinear growth. Along this class of geodesics we still
obtain that the map $s\mapsto \UU(\mu_s)$ is convex in $[0,1]$.
\end{remark}

\begin{remark}[The one-dimensional case]
\label{rem:dim1} \upshape When the space dimension $d=1$, then the
generalized McCann condition $\GMC\Mob1$ reduces to the usual
convexity of $U$. In this case, a simple approximation argument
shows that we can cover also the case of functions $U$ which are
not bounded in a right neighborhood of $0$ (and in a left
neighborhood of $M$, if $M<+\infty$) and the integrability
assumptions on $U''\Mob$ of \eqref{eq:81} and on $U''\Mob\Mob'$ of
\eqref{eq:20} can be dropped.
\end{remark}


\section{Action inequalities in the smooth case}\label{sec:smooth}
In this section we assume that $\Omega$ is a \emph{smooth and
bounded open set}. We consider a smooth curve
\begin{equation}
\label{eq:96} \mu_s:=\rho_s\Leb d\restr\Omega,\quad \rho\in
C^\infty([0,1]\times \bar\Omega),\quad 0<m_0\le \rho\le
m_1<M,\quad \mu_s(\Omega)\equiv \mass,\quad s\in [0,1].
\end{equation}
We also assume that $P$ and $\Mob$ are of class $C^\infty$ in
$[m_0,m_1]$. We consider the semigroup $\Semi=\Semi(P,\Omega)$
defined by Theorem \ref{thm:weak_classical} and we set
\begin{equation}
\label{eq:97} \mu_{s,t}:=\rho_{s,t}\Leb d\restr\Omega,\quad
\rho_{s,t}(\cdot)=\rho(s,t,\cdot):=\Semi_{st}\rho_s,\quad s\in
[0,1],\ t\ge0.
\end{equation}
Classical theory of quasilinear parabolic equation shows that
$\rho\in C^\infty([0,1]\times [0,\infty)\times\Omega)\cap
C^\infty([0,1]\times (0,+\infty)\times\bar\Omega).$

\newcommand{\rsigma}{\rho}

Since the semigroup $\Semi_t$ preserve the lower and upper bounds
on $\rho$ and $\int_\Omega \partial_s \rho\,\d x=0$, for every
$(s,t)\in [0,1]\times [0,+\infty)$ we can introduce the unique
solution $\zeta_{s,t}=\zeta(s,t,\cdot) \in C^\infty(\bar\Omega)$,
of the uniformly elliptic Neumann boundary value problem
\begin{equation}\label{EP}
\left\{\begin{array}{rl} -\nabla \cdot (\Mob(\rsigma)\nabla \zeta
) =\de_s \rsigma & \mbox{ in } \Omega, \\
\nabla \zeta \cdot \nn = 0 & \mbox{ on } \de\Omega, \\
\int_\Omega \zeta(x)\,\dx = 0. &
\end{array}\right.
\end{equation}
It is easy to check that $\zeta$ depends smoothly on $s$ and $t$.
Notice that \eqref{EP} is equivalent to
\begin{equation}
\label{eq:39} \int_\Omega \Mob(\rsigma)\,\nabla\zeta\cdot\nabla
\eta\,\dx= \int_\Omega \partial_s\rsigma\,\eta\,\dx\quad \forall\,
\eta\in C^1(\overline\Omega).
\end{equation}
By construction, for every $t\ge0$ the curve $s\mapsto
(\mu_{s,t},\nnu_{s,t})$ with
$\nnu_{s,t}:=\Mob(\rho_{s,t})\nabla\zeta_{s,t}\Leb d\restr\Omega$
belongs to $\CE(\mu_0\to\mu_{1,t})$ and its energy can be
evaluated by integrating the action
\begin{equation*}
\label{eq:95} A_{s,t}:=\Phi_{\Mob,\Omega}(\mu_{s,t},\nnu_{s,t})=
\int_\Omega \Mob(\rho_{s,t})|\nabla\zeta_{s,t}|^2\,\d x
\end{equation*}
with respect to $s$ in the interval $[0,1]$. The integral provides
an upper bound of the $\Wmod_{\Mob,\Omega}$-distance between
$\mu_0$ and $\mu_{1,t}=\rho_{1,t}\Leb d$, which corresponds to the
solution of the nonlinear diffusion equation \eqref{PME} with
initial datum $\rho_1$. As it was shown in \cite{DS}, evaluating
the time derivative of the action $A_{s,t}$ is a crucial step to
prove that \eqref{PME} satisfies the EVI formulation \eqref{eq:55}.
Next lemma, which does not require any convexity assumption on $\Omega$,
collects the main calculations.

\begin{lemma}\label{le:basic}
Let $\rho_s,\rho_{s,t},$ and $\zeta_{s,t}$ be as in \eqref{eq:96},
\eqref{eq:97}, and \eqref{EP}. Then for every
$(s,t)\in[0,1]\times(0,+\infty)$ we have
\begin{equation}\label{eq:basic}
\begin{aligned}
\frac 12\frac{\partial}{\partial t }A_{s,t}&= \frac{\partial
}{\partial t } \frac{1}{2}\int_\Omega |\nabla\zeta_{s,t}|^2
\Mob(\rsigma_{s,t})\,\dx
= -\int_\Omega \nabla P(\rsigma_{s,t})\cdot\nabla\zeta_{s,t} \,\dx \\
& - s \int_\Omega
\Big((P'(\rsigma_{s,t})\Mob(\rsigma_{s,t})-H(\rsigma_{s,t}))
(\Delta \zeta_{s,t})^2+
H(\rsigma_{s,t})|D^2\zeta_{s,t}|^2\Big)\,\dx\\
& + s \int_\Omega P'(\rsigma_{s,t})\Mob''(\rsigma_{s,t}) |\nabla
\rsigma_{s,t}|^2 |\nabla\zeta_{s,t}|^2\,\dx + \frac{1}{2}s
\int_{\de\Omega} H(\rsigma_{s,t}) \nabla|\nabla\zeta_{s,t}|^2
\cdot \nn\,d\HH^{d-1},
\end{aligned}
\end{equation}
where $H$ is defined in \eqref{eq:62}.
\end{lemma}

\begin{proof}
For keep the notation simple, we omit the explicit dependence of
$\rho,\zeta$ on $s,t$. By the definition of $\rsigma$ we easily
get
\begin{equation}\label{eq:19}
\de_t\rsigma = s \Delta P(\rsigma)\quad \text{and}\quad s
\int_\Omega \nabla P(\rsigma)\cdot\nabla \eta\,\dx= -\int_\Omega
\partial_t \rsigma\,\eta\,\dx\quad \forall\, \eta\in
C^1(\overline\Omega).
\end{equation}
Further differentiation with respect to $s$ yields
\begin{equation}
\label{eq:40} -\int_\Omega \nabla P(\rsigma)\cdot\nabla \eta\,\dx+
s\int_\Omega \partial_s P(\rsigma) \Delta\eta\,\dx= \int_\Omega
\partial_{st} \rsigma\,\eta\,\dx
\end{equation}
for all $\eta\in C^2(\overline\Omega)$ with $\nabla\eta\cdot\nn=0$
on $\partial\Omega$. On the other hand, differentiating
\eqref{eq:39} with respect to $t$ we obtain
\begin{equation}
\label{eq:41} \int_\Omega \Mob(\rsigma)\partial_t\nabla\zeta\cdot
\nabla\eta\,\dx= \int_\Omega \partial_{st}\rsigma\,\eta\,\dx-
\int_\Omega
\partial_t\Mob(\rsigma)\nabla\zeta\cdot\nabla\eta\,\dx\quad
\forall\, \eta\in C^1(\overline\Omega).
\end{equation}
The time derivative of the action functional is
\begin{align}
\notag &\frac{\partial}{\partial t }\frac{1}{2}\int_\Omega
|\nabla\zeta|^2\Mob(\rsigma)\,\dx = \frac{1}{2}\int_\Omega
\partial_t \Mob(\rsigma) |\nabla\zeta|^2\,\dx+ \int_\Omega
\de_t\nabla\zeta \cdot\nabla\zeta \Mob(\rsigma)\,\dx
\\\notag
&\topref{eq:41}=\int_\Omega \partial_{st}\rsigma \zeta\,\dx-
\frac{1}{2}\int_\Omega \partial_t \Mob(\rsigma) |\nabla\zeta|^2\,\dx\\
\label{eq:25}&\topref{eq:40}= -\int_\Omega \nabla
P(\rsigma)\cdot\nabla\zeta\,\dx+ s\int_\Omega \partial_s
P(\rsigma)\Delta\zeta\,dx - \frac{1}{2}\int_\Omega \partial_t
\Mob(\rsigma) |\nabla\zeta|^2\,\dx= I+II+III.
\end{align}
We evaluate separately the various contributions: concerning the
second integral II we introduce the auxiliary function $G$
\begin{equation}
\label{eq:43} G(r):=P'(r)\Mob(r)-H(r),\quad \text{so that}\quad
G'(r)=P''(r)\Mob(r),
\end{equation}
and we get
\begin{align}
II&=\notag s\int_\Omega \partial_s P(\rsigma)\Delta\zeta\,dx=
s\int_\Omega P'(\rsigma)\partial_s\rsigma\Delta\zeta\,dx
\notag\\&\topref{eq:39}= s\int_\Omega
P'(\rsigma)\Mob(\rsigma)\nabla\Delta\zeta\cdot\nabla\zeta\,\dx +
s\int_\Omega \Delta\zeta\,
P''(\rsigma)\Mob(\rsigma)\nabla\rsigma\cdot\nabla\zeta\,\dx
\notag\\&\topref{eq:43}= s\int_\Omega
P'(\rsigma)\Mob(\rsigma)\nabla\Delta\zeta\cdot\nabla\zeta\,\dx +
s\int_\Omega \Delta\zeta\, \nabla G(\rsigma)\cdot\nabla\zeta\,\dx
\notag\\
&= s\int_\Omega H(\rsigma)\nabla\Delta\zeta\cdot\nabla\zeta\,\dx -
s\int_\Omega G(\rsigma)\big(\Delta\zeta\big)^2\,\dx.\nonumber
\end{align}
The third integral of \eqref{eq:25} is
\begin{align}
\notag III&=- \frac{1}{2}\int_\Omega \Mob'(\rsigma)\partial_t
\rsigma |\nabla\zeta|^2\,\dx=
\\&\topref{eq:19}=
\frac{s}{2}\int_\Omega P'(\rsigma)\Mob''(\rsigma)|\nabla\rsigma|^2
|\nabla\zeta|^2\,\dx+ \frac s2 \int_\Omega
P'(\rsigma)\Mob'(\rsigma) \nabla|\nabla\zeta|^2
\cdot\nabla\rsigma\,\dx
\\&\topref{eq:62}=
\frac{s}{2}\int_\Omega P'(\rsigma)\Mob''(\rsigma)|\nabla\rsigma|^2
|\nabla\zeta|^2\,\dx + \frac s2 \int_\Omega \nabla|\nabla\zeta|^2
\cdot\nabla H(\rsigma)\,\dx.\nonumber
\end{align}
A further integration by parts in the last integral and the
Bochner formula
\begin{equation*}
\nabla\zeta\cdot\nabla \Delta\zeta-\frac{1}{2}\Delta
|\nabla\zeta|^2 = - |D^2 \zeta|^2
\end{equation*}
yield \eqref{eq:basic}.
\end{proof}

\begin{corollary}\label{lemma:main}
Under the same notation and assumptions of Lemma \ref{le:basic},
if $\Omega$ is convex and $U$ satisfies the generalized McCann
condition $\GMC\Mob d$ $(\ref{eq:gmcab}a,b)$, then
\begin{equation}\label{in}
\frac 12\frac \partial{\partial t}A_{s,t}=
\frac{\partial}{\partial t }\frac{1}{2}\int_\Omega
|\nabla\zeta_{s,t}|^2 \Mob(\rsigma_{s,t})\,\dx \leq
-\frac\partial{\partial s} \UU (\mu_{s,t})=
-\frac\partial{\partial s} \int_\Omega U(\rho_{s,t})\,\d x,
\end{equation}
and
\begin{equation}\label{iin}
\begin{aligned}
\frac 12\Wmod_{\Mob,\Omega}^2(\mu_{1,t},\mu_0)+t\UU(\mu_{1,t})\le
\frac{1}{2}\int_0^1 A_{s,t}\,\d s + \int_0^t \UU
(\mu_{1,\tau})\,\d\tau\leq \frac{1}{2}\int_0^1 A_{s,0}\,\d s + t
\UU (\mu_0).
\end{aligned}
\end{equation}
\end{corollary}

\begin{proof}
We determine the sign of the terms in the right hand side of
\eqref{eq:basic} thanks to (\ref{eq:gmcab}a,b) and the convexity
of $\Omega$. Recalling that $|D^2\zeta|\ge \frac 1d
(\Delta\zeta)^2$ and $H\ge0$ we obtain that the second integral in
the right--hand side of \eqref{eq:basic} is nonpositive
\begin{equation}\label{eq:101}
-\int_\Omega \Big((P'(\rsigma)\Mob(\rsigma)-H(\rsigma))(\Delta
\zeta)^2+ H(\rsigma)|D^2\zeta|^2\Big)\,\dx
\topref{hp:gmc}\leq 0.
\end{equation}
Since $P$ is increasing and $m$ is concave, we have
$P'(\rsigma)\Mob''(\rsigma)\leq 0$ which yields
\begin{equation}\label{eq:102}
\int_\Omega P'(\rsigma)\Mob''(\rsigma) |\nabla \rsigma|^2
|\nabla\zeta|^2\,\dx \leq 0.
\end{equation}
Since $H$ is nonnegative, the smoothness and convexity of $\Omega$
and the smoothness of $\zeta$ yields, see
\cite{Grisvard,Ott01,GST},
\begin{equation}
\nabla |\nabla \zeta|^2 \cdot \nn \leq 0, \quad \mbox{ on } \de
\Omega ,\quad \label{eq:103} \int_{\de\Omega} H(\rsigma)
\nabla|\nabla\zeta|^2 \cdot \nn\,d\HH^{d-1}\leq 0.
\end{equation}
Combining \eqref{eq:101}, \eqref{eq:102} and \eqref{eq:103},
\eqref{eq:basic} yields the inequality
\begin{equation}\label{eq:104}
\frac{\partial}{\partial t } \frac{1}{2}\int_\Omega
|\nabla\zeta|^2 \Mob(\rsigma)\,\dx \leq -\int_\Omega \nabla\zeta
\cdot \nabla P(\rsigma)\,\dx .
\end{equation}
On the other hand
\begin{equation}\label{eq:105}
\frac{\partial}{\partial s} \int_\Omega U(\rsigma) \,\d x =
\int_\Omega U'(\rsigma)\de_s \rsigma \,\dx \topref{eq:39} =
\int_\Omega \Mob(\rsigma) \nabla U'(\rsigma) \cdot\nabla \zeta
\,\dx = \int_\Omega \nabla P(\rsigma) \cdot\nabla\zeta \,\dx,
\end{equation}
so that \eqref{in} follows by \eqref{eq:104} and \eqref{eq:105}.
Integrating \eqref{in} with respect to $s$ and $t$,
we obtain the second inequality in \eqref{iin}.
The first inequality in \eqref{iin} follows from the definition of $\Wmod_{\Mob,\Omega}$
and the monotonicity of $\tau \mapsto \UU(\mu_{1,\tau})$ (see the energy identity \eqref{eq:120}).
\end{proof}

\section{Proof of the main theorems}
\subsection{The generation result}
Recall that $\GF_t(\mu_0)=\Semi_t (\rho_0)\Leb d\restr{\Omega}$
when $\mu_0=\rho_0\Leb d\restr\Omega$; Theorem \ref{th:main} is an
immediate consequence of the following result.
\begin{theorem}
\label{thm:generation} Let $\Omega$ be a bounded, convex open set of
$\Rd$ and let us assume that
$\mu_i\in \MM^+(\overline\Omega)$,
$i=0,1$, have finite distance
$\Wmod_{\Mob,\Omega}(\mu_0,\mu_1)<+\infty$ (and therefore the same
mass $\mass=\mu_0(\Omega)=\mu_1(\Omega)$) and satisfy
$\UU(\mu_0)<+\infty$ and $\mu_1=\rho_1\Leb d\restr\Omega \ll\Leb
d$. If $U$ satisfies the generalized McCann
condition $\GMC\Mob d$ $(\ref{eq:gmcab}a,b)$ then
\begin{equation*}
\label{eq:98} \frac12\Wmod_{\Mob,\Omega}^2(\GF_t \mu_1,\mu_0)+
t\UU(\GF_t \mu_1)\le \frac 12\Wmod_{\Mob,\Omega}^2(\mu_1,\mu_0)+t
\UU(\mu_0)\quad \forall\, t\ge0.
\end{equation*}
\end{theorem}
\begin{proof}
Since $\Wmod_{\Mob,\Omega}(\mu_0,\mu_1)<+\infty$ there exists a
geodesic curve $(\mu,\nnu)\in \CE_{\Omega}(\mu_0\to\mu_1)$ such
that
\begin{equation*}
\label{eq:99} \Phi_{\Mob,\Omega}(\mu_s,\nnu_s)\equiv
\int_0^1\Phi_{\Mob,\Omega}(\mu_s,\nnu_s)\,\d
s=\Wmod_{\Mob,\Omega}^2(\mu_0,\mu_1).
\end{equation*}
Applying Lemma \ref{le:smoothing}, we find a family of
approximating curves $(\mu^{\eps,\delta},\nnu^{\eps,\delta})$ and
smooth convex open sets $\Omega_\eps$ satisfying \eqref{eq:87},
\eqref{eq:90} and
\begin{equation*}
\label{eq:100} \int_0^1
\Phi_{\Mob,\Omega_\eps}(\mu^{\eps,\delta},\nnu^{\eps,\delta})\,\d
s \le c_\eps^2 \Wmod_{\Mob,\Omega}^2(\mu_0,\mu_1).
\end{equation*}
Let $\Mob_\eta,P_\eta$ as in Lemma \ref{le:smoothing_2}, for
constants $0<M'<M''<M$ such that $M'\le \rho^{\eps,\delta}\le M''$
in $\Omega_\eps$, and let $\zeta^{\eps,\delta,\eta}_{s}\in
C^\infty(\Rd)$ obtained by solving \eqref{EP} with respect to
$\Mob_\eta$ and $\partial_s \rho^{\eps,\delta}$ in $\Omega_\eps$.
Since $\Mob_\eta\ge \Mob$, by Theorem 5.21 of \cite{DNS} we easily
have
\begin{equation*}
\label{eq:107}
\int_{\Omega_\eps}{|\nabla\zeta_{s}^{\eps,\delta,\eta}|^2}{
\Mob_\eta(\rho^{\eps,\delta})}\,\d x =
\Phi_{\Mob_\eta,\Omega_\eps}(\mu^{\eps,\delta},\nnu^{\eps,\delta})\le
\Phi_{\Mob,\Omega_\eps}(\mu^{\eps,\delta},\nnu^{\eps,\delta})\quad
\forall\,\eta>0,\ s\in [0,1].
\end{equation*}
If $\GF^{\eps,\eta}=\GF(\UU_\eta,\Mob_\eta,\Omega_\eps)$ is the
semigroup associated with $S(P_\eta,\Omega_\eps)$ and the
corresponding integral functional $\UU_\eta$, \eqref{iin} gives
\begin{align*}
\frac12\Wmod^2_{\Mob_\eta,\Omega_\eps}(\GF^{\eps,\eta}_t\mu^{\eps,\delta}_1,\mu^{\eps,\delta}_0)+
t\UU_\eta(\GF^{\eps,\eta}_t\mu^{\eps,\delta}_1)&\le \frac
12\int_0^1\int_{\Omega_\eps}
{|\nabla\zeta_{s}^{\eps,\delta,\eta}|^2}{
\Mob_\eta(\rho^{\eps,\delta})}\,\d x\,\d s+
t\UU_\eta(\mu_0^{\eps,\delta})\\
&\le \frac{c_\eps^2}2 \Wmod^2_{\Mob,\Omega}(\mu_1,\mu_0)+
t\UU_\eta(\mu^{\eps,\delta}_0).
\end{align*}
Passing to the limit as $\eta\downarrow0$ (notice that the
functions $\Semi^{\eps,\eta}_t \rho^{\eps,\delta}$ take their
values in $[M',M'']$) we get
\begin{equation*}
\label{eq:109} \frac
12\Wmod_{\Mob,\Omega_\eps}^2(\GF_t^\eps\mu^{\eps,\delta}_1,\mu^{\eps,\delta}_0)+
t\UU(\GF^\eps_t\mu^{\eps,\delta}_1) \le \frac{c_\eps^2}2
\Wmod^2_{\Mob,\Omega}(\mu_1,\mu_0)+ t\UU(\mu^{\eps,\delta}_0)
\end{equation*}
where $\GF^\eps=\GF(\UU,\Mob,\Omega_\eps)$ is associated with
$\Semi(P,\Omega_\eps)$. We can then pass to the limit as
$\delta\downarrow0$: since $\rho^{\eps,\delta}\to\rho^\eps$ in
$L^\infty(\Omega_\eps)$ we immediately have
\begin{equation*}
\label{eq:127} \frac
12\Wmod_{\Mob,\Omega_\eps}^2(\GF_t^\eps\mu^{\eps}_1,\mu^{\eps}_0)+
t\UU(\GF^\eps_t\mu^{\eps}_1) \le \frac{c_\eps^2}2
\Wmod^2_{\Mob,\Omega}(\mu_1,\mu_0)+ t\UU(\mu^{\eps}_0).
\end{equation*}
Finally as $\eps\downarrow0$ we conclude, recalling Proposition
\ref{prop:stability}.
\end{proof}

\subsection{Geodesic convexity}
The proof of Theorem \ref{th:conv} follows immediately from the
generation result Theorem \ref{thm:generation} and Theorem
\ref{thm:crucial} if every measure $\mu_s$ of the geodesic curve
is absolutely continuous with respect to $\Leb d\restr\Omega$ (see
also Remark \ref{rem:conv_refine}). On the other hand, this
property is not known \emph{a priori}, so we need a more refined
argument.
\begin{proof}
As in Proposition \ref{prop:convolution} we set $\mu_s^\eps:=
\mu_s\ast k_\eps,\nnu_s^\eps:=\nnu_s\ast k_\eps$ and we denote by
$\GF^\eps=\GF(\UU,\Mob,\Omega_\eps)$.
By \eqref{convusc} and the contraction property given by Theorem \ref{th:main} in $\Omega_\eps$ we have
\begin{equation}
\label{eq:111}
\Wmod_{\Mob,\Omega_\eps}(\GF^\eps_t\mu^\eps_{s_1},\GF^\eps_t\mu^\eps_{s_2})
\le \Wmod_{\Mob,\Omega_\eps}(\mu^\eps_{s_1},\mu^\eps_{s_2})\le
\Wmod_{\Mob,\Omega}(\mu_{s_1},\mu_{s_2})=
|s_1-s_2|\Wmod_{\Mob,\Omega}(\mu_0,\mu_1)
\end{equation}
and by \eqref{convcont}, Theorem \ref{thm:crucial} and \eqref{eq:47} we have
\begin{equation*}
\label{eq:112}
\delta^2_\eps:=\Wmod_{\Omega,\Mob}^2(\mu_0,\mu_1)-\Wmod_{\Mob,\Omega_\eps}^2
(\mu_0^\eps,\mu_1^\eps)\to 0\quad\text{as }\eps\downarrow 0,
\end{equation*}
\begin{equation}
\label{eq:113}
\UU(\GF_t^\eps \mu_s^\eps)\le
(1-s)\UU(\mu^\eps_0)+s\UU(\mu^\eps_1)
+ \frac{\delta_\eps^2}{2 t}s(1-s),
\end{equation}
\begin{equation}
\label{eq:135}
\Wmod_{\Mob,\Omega_\eps}^2(\GF^\eps_t\mu_i^\eps,\mu_i)\le t\big(
\UU(\mu_i^\eps)-\inf \UU\big)\le t\big(\UU(\mu_i)-\inf\UU\big),
\end{equation}
where the second inequality in \eqref{eq:135} follows from Jensen's inequality.\\
We choose now a countable set $\mathscr C$ dense in $[0,1]$ and
containing $0$ and $1$, a vanishing sequence $(t_k)_{k\in \N}$ and
another vanishing sequence $(\eps_k)_{k\in\N}$ so that
$\lim_{k\uparrow+\infty} t_k^{-1}\delta_{\eps_k}^2=0$. By
compactness and a standard diagonal argument, up to extracting a
further subsequence, we can find limit points $\tilde\mu_s$ for
$s\in \mathscr C$ such that
\begin{equation*}
\label{eq:130} \GF^{\eps_k}_{t_k}\mu^{\eps_k}_s\weakto
\tilde\mu_{s} \quad\text{weakly as $k\uparrow+\infty$}.
\end{equation*}
By \eqref{eq:111} and Proposition \ref{prop:semi} we get
\begin{equation}
\label{eq:131}
\Wmod_{\Mob,\Omega}(\tilde\mu_{s_1},\tilde\mu_{s_2})\le
|s_1-s_2|\Wmod_{\Mob,\Omega}(\mu_0,\mu_1)\quad \forall\,
s_1,s_2\in \mathscr C.
\end{equation}
\eqref{eq:135} yields $\tilde\mu_0=\mu_0,\ \tilde\mu_1=\mu_1$ so
that we can extend $\tilde\mu$ to a continuous curve (still
denoted by $\tilde\mu$) connecting $\mu_0$ and $\mu_1$ still
satisfying \eqref{eq:131} for every $s_1,s_2\in [0,1]$. The
triangular inequality shows that \eqref{eq:131} is in fact an
equality and the curve $\tilde\mu$ is a constant speed minimizing
geodesic. On the other hand, the lower semicontinuity of $\UU$
with respect to weak convergence and \eqref{eq:113} yields
\begin{equation}
\label{eq:132} \UU(\tilde\mu_s)\le
(1-s)\UU(\mu_0)+s\UU(\mu_1)\quad \forall\, s\in \mathscr C.
\end{equation}
A further lower semicontinuity and density argument shows that
\eqref{eq:132} holds for every $s\in [0,1]$.
\end{proof}

\section*{Final remarks and open problems} This paper is a
first step towards the investigation of the geometry of spaces of
measures metrized by $\Wmod_{\Mob,\Omega}$, the induced convexity
notions for integral functionals and the corresponding generation
of gradient flows with applications to various nonlinear
evolutionary PDE's.

Since a sufficiently general theory is far from being developed
and understood, it is in some sense surprising that one can
reproduce in this setting the celebrated McCann convexity result.
On the other hand, many interesting and basic problems remain
open: here is just a provisional list.

\begin{itemize}
\item[-] At the level of the distance $\Wmod_{\Mob,\Omega}$ only
partial results on some basic properties (such as density of
regular measures or necessary and sufficient conditions ensuring
the finiteness of the distance), are known and a complete and
accurate picture is still missing.

\item[-] The situation is even less clear in the case of unbounded
domains: in this paper we restricted our attention to the bounded
case only.

\item[-] The study of other integral functionals is completely
open, as well as applications to different types of evolution
equations, like scalar conservation laws or nonlinear fourth order
equation of thin-film type, see the introduction of \cite{DNS}.

\item[-] It would be interesting to study other metric quantities
(e.g. the metric slope) and the pseudo-Riemannian structure
(tangent space, Alexandrov curvature, etc) connected with the
distance and the energy functionals, see \cite{AGS,CMV06}.

\item[-] The regularization properties and asymptotic behaviour of
the gradient flow $\UU$ and its perturbation can be studied as
well: in the Wasserstein case the geodesic convexity of a
functional yields many interesting estimates.

\item[-] The convergence of the so called ``Minimizing movement''
or JKO-scheme could be exploited in this and other situations: in
the case of geodesically convex energies, further information on
the Alexandrov curvature of the distance $\Wmod_{\Mob,\Omega}$
would be crucial, see \cite{AGS,Sa07}.
\end{itemize}


\noindent {\large\bf Acknowledgments.} J.A.C. acknowledges the
support from DGI-MICINN (Spain) project MTM2008-06349-C03-03.
G.S. and S.L. have been partially supported by MIUR-PRIN'06 grant for
the project ``Variational methods in optimal mass transportation and in
geometric measure theory''. D.S. was partially supported by NSF grant DMS-0638481. D.S. would also
like to thank the Center for Nonlinear Analysis (NSF grants DMS-0405343 and
DMS-0635983) for its support during the preparation of this paper.


\end{document}